\documentclass[reqno,12pt]{amsart}
\usepackage{amsmath,amssymb,amsthm}
\usepackage{multirow}
\usepackage{url}
\usepackage{mathtools}

\numberwithin{equation}{section}

\catcode`\@=11
\def\section{\@startsection{section}{1}%
 \z@{.7\linespacing\@plus\linespacing}{.5\linespacing}%
 {\normalfont\Large\bfseries\scshape\centering}}

\def\subsection{\@startsection{subsection}{2}%
  \z@{.5\linespacing\@plus\linespacing}{.5\linespacing}%
  {\normalfont\large\bfseries\scshape}}

\def\subsubsection{\@startsection{subsubsection}{3}%
 \z@{.5\linespacing\@plus\linespacing}{-.5em}
 {\normalfont\large\bfseries}}
\catcode`\@=12

\newtheorem{theorem}{Theorem}
\newtheorem{prop}[theorem]{Proposition}
\newtheorem{lemma}[theorem]{Lemma}
\newtheorem{corollary}[theorem]{Corollary}
\newtheorem{conjecture}[theorem]{Conjecture}

\theoremstyle{remark}
\newtheorem*{remark}{Remark}
\newtheorem*{remarks}{Remarks}

\newcommand{\Si}{S_{\hspace{-0.6pt}i}}
\newcommand{\Sj}{S_{\hspace{-1pt}j}}
\newcommand{\Sn}{S_{\hspace{-0.6pt}n}}

\def\al{\alpha}
\def\CT{\operatorname{CT}}
\def\coef#1{\left\langle#1\right\rangle}
\def\fl#1{\lfloor#1\rfloor}
\def\cl#1{\lceil#1\rceil}
\def\vv{\sqrt v}
\def\N{\mathbb{N}}
\def\ii{\mathbf i}
\def\po#1#2{(#1)_#2}

\begin{document}

\title[Determinant evaluations]
      {Determinant evaluations inspired by Di Francesco's determinant
for twenty-vertex configurations}

\author[C. Koutschan]{C. Koutschan}

\address{Johann Radon Institute for Computational and Applied Mathematics
  (RICAM),
Austrian Academy of Sciences,
Altenberger Stra\ss e 69,
A-4040 Linz.\newline
WWW: {\tt http://www.koutschan.de}.}

\author[C. Krattenthaler]{C. Krattenthaler}

\address{Fakult\"at f\"ur Mathematik, Universit\"at Wien,
Oskar-Morgenstern-Platz~1, A-1090 Vienna, Austria.
WWW: {\tt http://www.mat.univie.ac.at/\lower0.5ex\hbox{\~{}}kratt}.}

\author[M. J. Schlosser]{M. J. Schlosser}

\address{Fakult\"at f\"ur Mathematik, Universit\"at Wien,
Oskar-Morgenstern-Platz~1, A-1090 Vienna, Austria.
WWW: {\tt http://www.mat.univie.ac.at/\lower0.5ex\hbox{\~{}}schlosse}.}

\dedicatory{Dedicated to the memory of Marko Petkov\v sek,\\ who had a
keen interest in computer algebra and determinants}

\thanks{
The first two authors were
partially supported by the Austrian Science Fund FWF,
grant 10.55776/F50.
The first author was also partially supported by the Austrian Science Fund FWF,
grant 10.55776/I6130.
The third author was partially supported by the Austrian Science Fund FWF,
grant 10.55776/P32305.}



\begin{abstract}
In his work on the twenty vertex model, Di Francesco [{\it
    Electron. J. Combin.} {\bf28}(4) (2021), Paper No.~4.38] found a
determinant formula for the number of configurations in a specific
such model, and he conjectured a closed form product formula for the
evaluation of this determinant. We prove this conjecture here.
Moreover, we actually generalize this determinant evaluation to a
one-parameter family of determinant evaluations, and we present
many more determinant evaluations of similar type --- some proved,
some left open as conjectures.
\end{abstract}

\maketitle

\section{Introduction}

In \cite{DiFranGui,DiFran21}, Di Francesco and Guitter undertook an
enumerative study of configurations in the twenty vertex model and
set it in relation to an analogous enumerative study of domino
tilings of certain regions. Particular such regions that
Di Francesco considered in \cite{DiFran21} were coined by him
``Aztec triangles". He found that certain twenty vertex configurations
were equinumerous with domino tilings of such ``Aztec triangles".
Moreover, he established a determinantal formula for these common
numbers, and he observed that this determinant had apparently an
evaluation given by a closed form
product.\footnote{There is a small subtlety that needs to be pointed
out here: our definition of binomial coefficients is not the same as
Di Francesco's. To be precise, his convention is to put
$\binom \al p  = 0$ for $-1 \le \al < p$.
Thus, according to this convention, 
all the entries in row~0 of the matrix of which the determinant is taken
in~\eqref{eq:DiFran} would equal~1, while, with our convention~\eqref{eq:bin},
they are all equal to~2. Consequently, our right-hand side
in~\eqref{eq:DiFran} has an additional factor~2 compared to
\cite[Eq.~(8.1)]{DiFran21}.}

\begin{conjecture}[{\sc Di Francesco \cite[Conj.~8.1 + Th.~8.2]{DiFran21}}]
\label{conj:DiFran}
For all positive integers $n$, we have
\begin{equation} \label{eq:DiFran} 
  \det_{0\leq i,j\le n-1} \left(2^i \binom{i+2j+1}{2j+1} - \binom{i-1}{2j+1}\right) =
  2\prod_{i=1}^n \frac{2^{i-1} \, (4i-2)!}{(n+2i-1)!},
\end{equation}
where the binomial coefficient is defined by
\begin{equation} \label{eq:bin} 
\binom \al p=\begin{cases} \frac {\al(\al-1)\cdots(\al-p+1)}
       {p!},&\text{if }p\ge0,\\
       0,&\text{if }p<0.
\end{cases}
\end{equation}
\end{conjecture}

This caught the attention of the second author. Since he prefers
parameters in determinants in order to facilitate their evaluation,
he searched for a parametric generalization of \eqref{eq:DiFran}.
He successfully found such a generalization, and in addition a
companion identity.

\begin{conjecture}[CK$_2$] \label{conj:CK}
For all positive integers $n$, we have
\begin{multline} \label{eq:CK1}
\det_{0\le i,j\le n-1}\left(2^i\binom {x+i+2j+1}{2j+1}+\binom {x-i+2j+1}{2j+1}\right)
\\
=
2^{\binom n2+1}\prod _{i=0} ^{n-1}\frac {i!} {(2i+1)!}
\prod _{i=0} ^{\fl{n/2}}(x+4i+1)_{n-2i}
\prod _{i=0} ^{\fl{(n-1)/2}}(x-2i+3n)_{n-2i-1},
\end{multline}
and
\begin{multline} \label{eq:CK2}
\det_{0\le i,j\le n-1}\left(2^i\binom {x+i+2j}{2j}+\binom {x-i+2j}{2j}\right)
\\
=
2^{\binom n2+1}\prod _{i=0} ^{n-1}\frac {i!} {(2i)!}
\prod _{i=0} ^{\fl{(n-1)/2}}(x+4i+3)_{n-2i-1}
\prod _{i=0} ^{\fl{(n-2)/2}}(x-2i+3n-1)_{n-2i-2},
\end{multline}
with the Pochhammer symbol $(\al)_p$ defined by
$\al(\al+1)\cdots(\al+p-1)$ for $p\ge1$ and $(\al)_0:=1$.
\end{conjecture}

Indeed, the special case $x=0$ of \eqref{eq:CK1} is equivalent
with~\eqref{eq:DiFran}. However, this did not help:
neither was the second author able to
prove Conjecture~\ref{conj:DiFran} nor was he able to prove
Conjecture~\ref{conj:CK} at the time.

In an unrelated development,
during the 9th International Conference on ``Lattice Path
Combinatorics and Applications" that took place June 21--25, 2021
at the CIRM in Luminy,
the first author\footnote{who did not participate in the conference;
the second and third author did (via {\tt zoom}).}
received the following e-mail from Doron Zeilberger.
More precisely, the e-mail arrived on June 23 at 13:49 (European time) and
had the subject ``challenge'':

\medskip
\begin{quote}
\tt  Dear Christoph,
  \medskip

  \noindent\raggedright
  Philippe Di Francesco just gave a great talk at the Lattice path conference
  mentioning, inter alia, a certain conjectured determinant. It is\\
  
  Conj.~8.1 (combined with Th.~8.2) in
  
  \url{https://arxiv.org/pdf/2102.02920.pdf}
  \medskip

  \noindent
  I am curious if you can prove it by the Koutschan-Zeilberger-Aek holonomic
  ansatz method.
  If you can do it before Friday, June 25, 2021, 17:00 Paris time, I will
  mention it in my talk in that conference.
  \medskip

  \noindent
  Best wishes\\
  \hspace*{50pt} Doron
\end{quote}

\medskip
\noindent
It turned out that the conjectured determinant evaluation in question ---
namely~\eqref{eq:DiFran} ---
was indeed routinely provable by the so-called
{\it holonomic Ansatz}. Consequently, the challenge was met and the result was
announced by Zeilberger in his talk on the last day of the
conference. 

Obviously, the second author immediately contacted the first. However,
the determinant evaluations in Conjecture~\ref{conj:CK} could not be
proved by the holonomic Ansatz.

On the other hand, the third author initiated a ``hunt" for further
determinant evaluations of similar kind. As a result, together we came
up with many variations of the determinant identities in
Conjectures~\ref{conj:DiFran} and~\ref{conj:CK}, some of which could
be proved by the holonomic Ansatz while others resisted to this method.
On the other hand, at least in one case a different --- non-algorithmic
--- method led to success.

Again in an unrelated development, at the Workshop on ``Enumerative
Combinatorics" in Oberwolfach in December 2022, Sylvie Corteel
asked the second author how to 
enumerate certain tableaux that were a mixture of
symplectic and supersymmetric tableaux; she and her student
Frederick Huang had
observed that their number seemed to be given by a nice product
formula. The second author asked for a bit of time, and answered
on the next day: ``Calculer un \hbox{d\'eterminant~!"} That was actually not
new for Corteel and Huang, they already knew that~\dots\ In any case,
a month later this determinant (and a related one) was indeed
evaluated; see~\cite{CoHuKr23}.\footnote{As it turned out, the tableau
enumeration problem was not so unrelated: the actual goal of Corteel
and Huang was to count the number of domino tilings of regions that
generalized Di Francesco's Aztec triangles; the ``super-symplectic"
tableaux were in bijection with these domino tilings.}
The second author noticed that, up to
a simple parameter transformation, the {\it result\/} of that determinant
evaluation seemed to be the same as the right-hand side of~\eqref{eq:CK1}
(and the {\it result\/} of the related determinant evaluation seemed to be the
same as the right-hand side of~\eqref{eq:CK2}). It did not take
for long to rigorously relate this determinant to the one on the
left-hand side of~\eqref{eq:CK1} (and the related determinant to the
one on the left-hand side of~\eqref{eq:CK2}). Thus, also
Conjecture~\ref{conj:CK} became a theorem.

\medskip
The purpose of this paper is to collect all these results and conjectures,
together with our proofs (in case we found one). More precisely, in the next
section we review Zeilberger's holonomic Ansatz.
Then follows a ``warmup" section, in which we prove a variation
of~\eqref{eq:CK1} in which the terms $2j$ in the binomial coefficients
get ``replaced" by~$j$ and the
power~$2^i$ is replaced by an arbitrary power~$a^i$. We actually
provide two proofs: one using the holonomic Ansatz, the other using
constant term calculus. Here, in this simple case, we are able to
display the results of the intermediate calculations when using the
holonomic Ansatz, while, due to
their size, this is not possible anymore for the
subsequent applications of the holonomic Ansatz in this
paper.\footnote{Instead, we provide details of these calculations
in the accompanying electronic material~\cite{EM}.}
In this sense, this proof also serves pedagogical purposes.

Section~\ref{sec:conj1} is devoted to the
earlier mentioned computer proof of Conjecture~\ref{conj:DiFran}
due to the first author that was announced by Zeilberger during
the 9th Lattice Path Conference. The proof of Conjecture~\ref{conj:CK}
is the subject of Section~\ref{sec:conj2}. As indicated above, the
idea of the proof is to relate the two determinants in~\eqref{eq:CK1}
and~\eqref{eq:CK2} to two determinants that had been evaluated
in~\cite{CoHuKr23}.

The subsequent sections discuss variations of
these determinant evaluations. We begin in Section~\ref{sec:var1} with
determinants of the kind as in Conjecture~\ref{conj:CK} where we allow
more general shifts at several places; see the definition of
$D_{\alpha,\beta,\gamma,\delta}(n)$ at the beginning of the section.
A computer search led to the discovery of many (more) corresponding
determinant evaluations; see Theorem~\ref{thm:det22}.

In Section~\ref{sec:var2}, we consider variations of the determinants in
Conjecture~\ref{conj:CK} in which $2j$ gets ``replaced" by~$3j$ and
the power~$2^i$ is replaced by~$3^i$, again allowing more general
shifts. Also in this case, we found many
corresponding determinant evaluations; see Theorem~\ref{thm:det33}.
Moreover, it seems that there are three one-parameter families of
closed-form determinant evaluations of this type; see
Conjecture~\ref{conj:det33x}.

Section~\ref{sec:var3} is dedicated to
variations of the determinants in
Conjecture~\ref{conj:CK} in which the power~$2^i$ gets ``replaced''
by~$4^i$, with the $2j$ in the binomial coefficients
being retained, again allowing shifts.
The ``sporadic" determinant evaluations that we found are listed
(and proved) in Theorem~\ref{thm:det24}. There is also a
one-parameter family of such evaluations; see Theorem~\ref{thm:detx41}.
Moreover, we discovered a second one-parameter family. However, in
that case the result of the determinant evaluation does not factor
completely. Our result in Theorem~\ref{thm:MS1} identifies all factors
but one, apparently, irreducible factor. While we failed to find an
explicit formula for that factor, we found a recurrence that it
seems to satisfy; see Conjecture~\ref{conj:MS1}. Since our
(non-algorithmic) proof of Theorem~\ref{thm:MS1} is somewhat lengthy,
it is given separately in Section~\ref{sec:proof}. 

Section~\ref{sec:var5} contains yet further
variations of the determinants in Conjecture~\ref{conj:CK}: here, the
power~$2^i$ remains untouched, but $2j$ gets ``replaced" by~$4j$,
and we allow more general shifts. We present our corresponding
findings in Conjecture~\ref{conj:4j}, Proposition~\ref{prop:4j}, and
Conjecture~\ref{conj:det42}. Again, we are confident that the holonomic
Ansatz is able to prove all these results. However, at this point in time
the capacity of the available computers is not sufficient to actually
carry out the necessary computations.

In the final section,
Section~\ref{sec:open}, we list several problems
left open or posed by this work.

\section{The Holonomic Ansatz}
\label{sec:holonom}

The \emph{holonomic Ansatz}~\cite{Zeilberger07} is a computer-algebra-based
approach to find and/or prove the evaluation of a symbolic determinant
$\det(A_n)$, where the dimension of the square matrix $A_n:=(a_{i,j})_{0\leq
  i,j<n}$ is given by a symbolic parameter~$n$. The method is only applicable
to non-singular matrices whose entries~$a_{i,j}$ are holonomic sequences (see
below) in the index variables~$i$ and~$j$. Moreover, the entries~$a_{i,j}$
must not depend on~$n$, i.e., $A_{n-1}$ is an upper-left submatrix of~$A_n$.

The holonomic Ansatz works as follows: define the quantity
\begin{equation}\label{cnj}
  c_{n,j} := (-1)^{n-1+j} \frac{M_{n-1,j}}{M_{n-1,n-1}}
\end{equation}
where $M_{i,j}$ denotes the $(i,j)$-minor of the matrix~$A_n$ (where the
indexing starts at~$0$). In other words, $c_{n,j}$ is the $(n-1,j)$ cofactor
of $A_n$ divided by $\det(A_{n-1})$. Using Laplace expansion with respect to
the last row, one can write
\begin{equation}\label{H3}\tag{H3}
  \sum_{j=0}^{n-1} a_{n-1,j} c_{n,j} = \frac{\det(A_n)}{\det(A_{n-1})}.
\end{equation}
Under the assumptions that (i) the bivariate sequence $c_{n,j}$ is holonomic
and that (ii) its holonomic definition is known, the symbolic sum on the
left-hand side of~\eqref{H3} can be tackled with creative
telescoping~\cite{Zeilberger91,PetkovsekWilfZeilberger96},
yielding a linear recurrence in~$n$ for the
sum. If a conjectured evaluation~$b_n$ for the determinant of~$A_n$ is
available, then one can prove it by verifying that $b_n/b_{n-1}$ satisfies the
obtained recurrence and by comparing a sufficient number of initial values. If
in contrast such a conjecture has not been formulated, then one may succeed to
find (and at the same time: prove) an evaluation of $\det(A_n)$ by solving the
recurrence, thus obtaining an expression for $\det(A_n)/\det(A_{n-1})$, and by
taking the product.

What can be said about the two assumptions? There is no general theorem that
implies that $c_{n,j}$ is always holonomic, and in fact, there are many
examples where it is not. If (i) is not satisfied, i.e., if $c_{n,j}$ is not
holonomic, then the method fails (not necessarily; in some situations one may
succeed to overcome the problem by applying a mild
reformulation; see~\cite{KoutschanNeumuellerRadu16}). Concerning~(ii): by a holonomic
definition we mean a set of linear recurrence equations whose coefficients
are polynomials in the sequence indices~$n$ and~$j$, together with finitely
many initial values, such that the entire bivariate sequence
$(c_{n,j})_{1\leq n,\,0\leq j<n}$ can be produced by unrolling the recurrences
and by using the initial values. The question now is how the original
definition~\eqref{cnj} can be converted into a holonomic definition.

Clearly, \eqref{cnj} allows one to compute the values of $c_{n,j}$ for
concrete integers~$n$ and~$j$ in a certain, finite range. From these data,
candidate recurrences can be constructed by the method of guessing (i.e.,
employing an Ansatz with undetermined coefficients; cf.~\cite{Kauers09}). It
remains to prove that these recurrences, constructed from finite, and
therefore incomplete data, are correct, i.e., are valid for all $n\geq1$ and
$0\leq j<n$. For this purpose, we show that $c_{n,j}$ is the unique solution
of a certain system of linear equations, and then we prove that the sequence defined by
the guessed recurrences (and appropriate initial conditions) also satisfies the
same system. By uniqueness, it follows that the two sequences agree,
i.e., that the guessed recurrences define the desired sequence~$c_{n,j}$.

Suppose that the last row of $A_n$ is replaced by its $i$-th row;
the resulting matrix is clearly singular, turning~\eqref{H3} into
\begin{equation}\label{H2}\tag{H2}
  \sum_{j=0}^{n-1} a_{i,j} c_{n,j} = 0\qquad (0\leq i<n-1).
\end{equation}
For each $n\in\N$ the above equation~\eqref{H2} represents a system
of $n-1$ linear equations in the $n$ ``unknowns'' $c_{n,0},\dots,c_{n,n-1}$,
whose coefficient
matrix $(a_{i,j})_{0\leq i<n-1,0\leq j<n}$ has full rank because
$\det(A_{n-1})\neq0$ (if the latter is not known a priori, it can be argued by
induction on~$n$). Hence the homogeneous system~\eqref{H2} has a
one-dimensional kernel. The solution is made unique by normalizing with respect to its
last component, that is, by imposing a condition that is obvious
from~\eqref{cnj}, namely
\begin{equation}\label{H1}\tag{H1}
  c_{n,n-1} = 1.
\end{equation}
Hence, \eqref{H1} and~\eqref{H2} together define $c_{n,j}$ uniquely. On the
other hand, given a holonomic definition of $c_{n,j}$, creative telescoping
and holonomic closure properties can be applied to prove~\eqref{H1}
and~\eqref{H2}, respectively. If these proofs succeed, then it follows that
the guessed recurrences are correct.

\medskip
The holonomic Ansatz has already been applied in many different contexts
\cite{KoutschanKauersZeilberger11, KoutschanThanatipanonda19,
  DuKoutschanThanatipanondaWong22}. Variations of it have been described in
\cite{IshikawaKoutschan12, KoutschanThanatipanonda13,
  KoutschanNeumuellerRadu16}.

\medskip
We conclude this introduction to the holonomic Ansatz with some remarks
concerning its concrete implementation. For computing the data, i.e., the
values of~$c_{n,j}$, it is usually more efficient to employ their definition
via~\eqref{H1} and~\eqref{H2}, rather than computing determinants in the
spirit of~\eqref{cnj}. We used the {\sl Mathematica} packages
\texttt{Guess.m}~\cite{Kauers09} for the guessing of the recurrences, and
\texttt{HolonomicFunctions.m}~\cite{Koutschan10b} for the creative-telescoping
proofs.

\section{A warmup exercise}
\label{sec:warmup}

Before we dedicate ourselves to the proofs of the conjectured
determinant evaluations of the introduction, we begin with a variation
of the determinants appearing in Conjecture~\ref{conj:CK}. The
variation consists in ``replacing" $2j$ in the binomial coefficients
by~$j$ and the power~$2^i$ by
$a^i$ where $a$ is an indeterminate. It turns out that a proof of
the evaluation
of this latter determinant is much simpler. We provide actually two proofs:
one using the holonomic Ansatz, and the other using constant term
calculus. (If one wishes: a computer proof and a computer-free
proof.) We will use this determinant evaluation later in the proof of
Theorem~\ref{thm:MS1} in Section~\ref{sec:proof}.

\begin{theorem} \label{thm:einfach}
For all non-negative integers $n$, we have
$$
\det_{0\le i,j\le n-1}\left(a^i\binom {x+i+j-1}j+\binom {x-i+j-1}j\right)
=2(a-1)^{\binom n2}.
$$
\end{theorem}

\begin{proof}[First proof]
We compute the \texttt{data} for $c_{n,j}$, as defined in \eqref{cnj}, for $1\leq n\leq 11$:
  \begin{alignat}{3}
    c_{1,0} &= 1, \label{eq:det1a_c1}
    \\
    c_{2,0} &= -x, &
    c_{2,1} &= 1,  \label{eq:det1a_c2}
    \\
    c_{3,0} &= \frac{a x^2+a x-x^2+x}{2 (a-1)}, &
    c_{3,1} &= \frac{-ax-a+x}{a-1}, &
    c_{3,2} &= 1, \\[-1ex]
    \vdots\notag \\[-1ex]
    c_{11,0} &= \frac{362880 x + \dots +a^9 x^{10}}{3628800 (a-1)^9},\ \dots,\ &
    c_{11,9} &= \frac{-ax-9 a+x}{a-1},\quad &
    c_{11,10} &= 1.\notag
  \end{alignat}
  Then we use the \texttt{Guess.m} package~\cite{Kauers09} to find plausible
  candidates for bivariate recurrences that $c_{n,j}$ may satisfy:
  \begin{center}
    \verb|g = GuessMultRE[data, {c[n,j], c[n,j+1], c[n+1,j], c[n+1,j+1]},| \\
    \verb|  {n, j}, 2, StartPoint -> {1, 0}, Constraints -> (j < n)];|
  \end{center}
  In order to have a canonical set of generators for the infinite set of
  such recurrences, which is a left ideal in the corresponding operator algebra,
  also called the \emph{annihilator} of the sequence $c_{n,j}$, we compute a
  (left) Gr\"obner basis \texttt{annc} of the previous output:
  \begin{center}
    \verb|OreGroebnerBasis[NormalizeCoefficients /@ ToOrePolynomial[g, c[n,j]]];|
  \end{center}
  As a result, we obtain the following two recurrences, which, in contrast to the
  recurrences in later sections, are small enough to be displayed here, albeit
  too unhandy to process them by pencil and paper:
\begin{align*}
& (1-a) n (j-n) \bigl(a j^2+2 a j x+a j+a x^2+a x-x^2-x\bigr) c_{n+1,j} \\
&  -(j-n+x+2) \bigl(a j^3-2 a j^2 n+2 a j^2 x+a j^2-4 a j n x\\
&\kern6cm
  +a j x^2+a j x-2 a n x^2 
     -j x^2 +j x+2 n x^2\bigr) c_{n,j+1} \\
   & +\bigl(a^2 j^2 n^2+a^2 j^2 n x-a^2 j^2 n+2 a^2 j n^2 x+a^2 j n^2+2 a^2 j n x^2-a^2 j n x
     -a^2 j n \\
   &\quad +a^2 n^2 x^2+a^2 n^2 x+a^2 n x^3-a^2 n x+a j^4-4 a j^3 n+2 a j^3 x+4 a j^3+3 a j^2 n^2 \\
   &\quad -9 a j^2 n x-9 a j^2 n+a j^2 x^2+5 a j^2 x+5 a j^2+6 a j n^2 x+3 a j n^2-6 a j n x^2 \\
   &\quad -11 a j n x-5 a j n+a j x^2+3 a j x+2 a j+2 a n^2 x^2-2 a n x^3-4 a n x^2-j^2 x^2 \\
   &\quad +j^2 x +4 j n x^2-j x^2+j x-3 n^2 x^2-n^2 x+n x^3 +4 n x^2+n x\bigr)c_{n,j} = 0,
   \\[1ex]
   & \bigl(a j^2+2 a j x+a j+a x^2+a x-x^2-x\bigr) (j-n+x+3) c_{n,j+2} \\
   & +\bigl(a^2 j^3+3 a^2 j^2 x+3 a^2 j^2+3 a^2 j x^2+6 a^2 j x+2 a^2 j+a^2 x^3
     +3 a^2 x^2+2 a^2 x-2 a j^3 \\
   &\quad +2 a j^2 n-5 a j^2 x-8 a j^2+4 a j n x+4 a j n-5 a j x^2
     -14 a j x-8 a j+2 a n x^2+2 a n x \\
   &\quad -2 a x^3-8 a x^2-6 a x+2 j x^2+2 j x-2 n x^2-2 n x+x^3+5 x^2 +4 x\bigr) c_{n,j+1} \\
   & -(a-1) (j-n+1) \bigl(a j^2+2 a j x+3 a j+a x^2+3 a x+2 a-x^2-x\bigr) c_{n,j} = 0.
  \end{align*}
  Next, we have to prove the identities \eqref{H1} and \eqref{H2}, in order to justify that
  $c_{n,j}$, as defined in \eqref{cnj}, agrees with the unique solution of the above
  recurrences, or in other words, that these guessed recurrences are correct.
  The command
  \begin{center}
    \verb|DFiniteSubstitute[annc, {j -> n-1}]|
  \end{center}
  delivers the following, second-order recurrence for $c_{n,n-1}$:
  \begin{align*}
    & (n+1)  \bigl(a n^2+2 a n x-a n+a x^2-a x-x^2+x\bigr) c_{n+2,n+1} \\
    & +\bigl(a^2 n^3+3 a^2 n^2 x+3 a^2 n x^2-a^2 n+a^2 x^3-a^2 x-2 a n^3-5 a n^2 x-5 a n x^2\\
    &\qquad +2 a n-2 a x^3+a x^2+a x+2 n x^2-2 n x+x^3-x^2\bigr) c_{n+1,n} \\
    & -(a-1) (n+x-1) \bigl(a n^2+2 a n x+a n+a x^2+a x-x^2+x\bigr) c_{n,n-1} =0.
   \end{align*}
  It is easy to check that the constant solution $c_{n,n-1}=1$ is a solution
  to the above recurrence, which, together with the initial conditions
  from~\eqref{eq:det1a_c1} and~\eqref{eq:det1a_c2}, implies~\eqref{H1}.

  In order to prove~\eqref{H2}, we view $c_{n,j}$ as a trivariate sequence in $n,i,j$,
  and compute the annihilator of $\binom{x-i+j-1}{j}\cdot c_{n,j}$ via closure properties:
  \begin{center}
    \verb|s1 = DFiniteTimes[Annihilator[Binomial[x-i+j-1,j], {S[n], S[j], S[i]}],| \\
    \verb|OreGroebnerBasis[Append[annc, S[i]-1], OreAlgebra[S[n], S[j], S[i]]]];|
  \end{center}
  Since we have a recursive definition of the summand, we can employ creative telescoping
  to find a set of recurrences that is satisfied by the sum
  $s^{(1)}_{n,i} = \sum_{j=0}^{n-1} \binom{x-i+j-1}{j}\cdot c_{n,j}$
  \begin{center}
    \verb|ct1 = FindCreativeTelescoping[s1, S[j]-1];|
  \end{center}
  and similarly for the other sum $s^{(2)}_{n,i}=\sum_{j=0}^{n-1} a^i \binom{x+i+j-1}{j}\cdot c_{n,j}$.
  Combining the two results via the command
  \begin{center}
    \verb|DFinitePlus[ct1[[1]], ct2[[1]]];|
  \end{center}
  yields recurrences for the sum $s_{n,i}=s^{(1)}_{n,i}+s^{(2)}_{n,i}$ on the left-hand side of~\eqref{H2}:
  \begin{align*}
    & (a-1) (i+1) n s_{n+1,i}-2 i (i-n+2) s_{n,i+1} + (a+1) (i+1) (i+n-1) s_{n,i} = 0,
    \\[1ex]
    & 2 i (i+1) (i-n+4) \bigl(a i^2+a i-2 a-x^2-x\bigr) s_{n,i+3} \\
    & -i \bigl(2 a^2 i^4+10 a^2 i^3+10 a^2 i^2-10 a^2 i-12 a^2+a i^5
      -a i^4 n-a i^4 x+8 a i^4+a i^3 n x -3 a i^3 n \\
    &\quad -5 a i^3 x+18 a i^3+2 a i^2 n x+a i^2 n-2 a i^2 x^2
      -5 a i^2 x+4 a i^2-3 a i n x+3 a i n-8 a i x^2 \\
    &\quad +a i x-19 a i-6 a x^2-6 a x-12 a-i^3 x^2-i^3 x+i^2 n x^2
      +i^2 n x+i^2 x^3-6 i^2 x^2-7 i^2 x \\
    &\quad -i n x^3+i n x^2+2 i n x+4 i x^3-11 i x^2-15 i x-n x^3+n x+3 x^3-6 x^2-9 x\bigr)
      s_{n,i+2} \\
    & + i \bigl(a^2 i^5-a^2 i^4 x+7 a^2 i^4+2 a^2 i^3 n-4 a^2 i^3 x+13 a^2 i^3
      +8 a^2 i^2 n-a^2 i^2 x-3 a^2 i^2+2 a^2 i n \\
    &\quad +6 a^2 i x-18 a^2 i-12 a^2 n+a i^5\!-a i^4 x+5 a i^4-a i^3 x^2-5 a i^3 x
      +5 a i^3+a i^2 x^3-5 a i^2 x^2 \\
    &\quad -7 a i^2 x-5 a i^2-2 a i n x^2-2 a i n x+3 a i x^3-8 a i x^2-5 a i x
      -6 a i-6 a n x^2+2 a x^3\!-2x \\
    &\quad -6 a n x -2 a x-i^3 x^2-i^3 x+i^2 x^3-3 i^2 x^2-4 i^2 x+3 i x^3
      -2 i x^2-5 i x+2 x^3\bigr) s_{n,i+1} \\
    & -a (i+2) (i+n-1) (i-x+1) \bigl(a i^3+2 a i^2-3 a i-i x^2-i x-x^2-x\bigr) s_{n,i} = 0.
  \end{align*}
  The sequence $s_{n,i}$ is restricted to $0\leq i<n-1$, and thus the support of the
  above recurrences prohibits one to use them for computing
  $s_{2,0}, s_{3,0}, s_{3,1}, s_{4,1}, s_{4,2}, s_{5,2}$; these have to be given as
  initial values. Moreover, one cannot use the second recurrence for computing $s_{n,3}$
  due to the factor~$i$ in its leading coefficient. This forces us to also include
  $s_{5,3}$ and $s_{6,3}$ into the initial conditions (note that $s_{n,3}$ for $n\geq7$
  can be computed using the first recurrence). It is not difficult to verify that all
  eight initial conditions are zero, and by virtue of the recurrences satisfied by
  $s_{n,i}$, it follows that $s_{n,i}=0$ for all $n,i$ with $0\leq i<n-1$.

  Identity \eqref{H3} is proven in a similar way. The sum on its left-hand side is split
  into two sums. A recurrence for the first one is obtained by calling
  \begin{center}
    \verb|ct1 = FindCreativeTelescoping[DFiniteTimes[|\hskip 160pt\null \\
    \hskip 50pt \verb|Annihilator[Binomial[x-n+j,j], {S[n], S[j]}], annc], S[j]-1];|
  \end{center}
  An analogous computation is done for the second sum. Combining the two results
  via the command
  \begin{center}
    \verb|DFinitePlus[ct1[[1]], ct2[[1]]];|
  \end{center}
  yields the following recurrence for the sum $s_n=\sum_{j=0}^{n-1} a_{n,j} c_{n,j}$:
  \[
    (a-1) n s_{n+2} - (a^2 n-6 a n+2 a+n) s_{n+1}-2 (a-1) a (2 n-1) s_n = 0.
  \]
  It is readily checked that $(a-1)^{\binom{n}{2}-\binom{n-1}{2}}=(a-1)^{n-1}$ is a solution
  of this recurrence, and that the necessary initial values are correct (i.e., that
  the asserted determinant evaluation holds for $n\leq3$). This concludes the proof
  of \eqref{H3}, and therefore the proof of the whole theorem.
\end{proof}

\begin{proof}[Second proof]
We use that $\binom Nk=\CT_{z}(1+z)^Nz^{-k}$, where $\CT_zf(z)$
denotes the constant term in~$z$ in the Laurent series $f(z)$. 
Furthermore, for a
Laurent aeries $f(z_0,z_1,\dots,z_{n-1})$ 
in $z_0,z_1,\dots,z_{n-1}$, we shall use the short notation
$$\CT_{\mathbf z}f(z_0,z_1,\dots,z_{n-1})$$ 
to denote the constant
term in this Laurent series.

Using these notations, our determinant can be written as
\begin{align*}
\det_{0\le i,j\le n-1}&\left(a^i\binom {x+i+j-1}j+\binom {x-i+j-1}j\right)\\
&=\CT_{\mathbf z}
\det_{0\le i,j\le n-1}\left(a^i\frac {(1+z_j)^{x+i+j-1}} {z_j^j}
+\frac {(1+z_j)^{x-i+j-1}} {z_j^j}\right)\\
&=
\CT_{\mathbf z}a^{\frac {1} {2}\binom n2}
\bigg(\prod _{j=0} ^{n-1}\frac {(1+z_j)^{x+j-1}} {z_j^j}\bigg)
\det_{0\le i,j\le n-1}\left(a^{i/2}(1+z_j)^{i}
+a^{-i/2}(1+z_j)^{-i}\right).
\end{align*}
The determinant can be evaluated by means of \cite[Eq.~(2.5)]{KratBN}.
Thus, we obtain
\begin{align*}
\det_{0\le i,j\le n-1}&\left(a^i\binom {x+i+j-1}j+\binom {x-i+j-1}j\right)\\
&\kern-.5cm
=
\CT_{\mathbf z}2a^{-\frac {1} {2}\binom n2}
\bigg(\prod _{j=0} ^{n-1}\frac {(1+z_j)^{x+j-n}} {z_j^j}\bigg)\\
&\kern1cm
\times
\bigg(
\prod _{0\le i<j\le n-1} ^{}
\left(\sqrt a(1+z_i)-\sqrt a(1+z_j)\right)
\left(1-a(1+z_i)(1+z_j)\right)
\bigg)\\
&\kern-.5cm
=
\CT_{\mathbf z}2
\bigg(\prod _{j=0} ^{n-1}\frac {(1+z_j)^{x+j-n}} {z_j^j}\bigg)
\bigg(
\prod _{0\le i<j\le n-1} ^{}
\left(z_i-z_j\right)
\big((1-a) -a(z_i+z_j+z_iz_j)\big)
\bigg).
\end{align*}
Since this is a constant term, we get the same value if we permute the
variables $z_0,z_1,\dots, z_{n-1}$. So, let us symmetrize the last
expression, meaning that we sum this expression over all possible
permutations of the variables. Obviously, in order to get the same value
again, we must divide the result by~$n!$. This leads to
\begin{multline*}
\det_{0\le i,j\le n-1}\left(a^i\binom {x+i+j-1}j+\binom {x-i+j-1}j\right)\\
=\frac {2} {n!}
\CT_{\mathbf z}
\bigg(\prod _{j=0} ^{n-1}(1+z_j)^{x-n}\bigg)
\bigg(
\prod _{0\le i<j\le n-1} ^{}
\left(z_i-z_j\right)
\big((1-a) -a(z_i+z_j+z_iz_j)\big)
\bigg)\\
\times
\det_{0\le i,j\le n-1}\left(\left(\frac {1+z_i} {z_i}\right)^j\right).
\end{multline*}
The determinant can be evaluated by means of the evaluation of the
Vandermonde determinant, so that
\begin{align*}
\det_{0\le i,j\le n-1}&\left(a^i\binom {x+i+j-1}j+\binom {x-i+j-1}j\right)\\
&\kern-13pt
=\frac {2} {n!}
\CT_{\mathbf z}
\bigg(\prod _{j=0} ^{n-1}(1+z_j)^{x-n}\bigg)
\bigg(
\prod _{0\le i<j\le n-1} ^{}
\left(z_i-z_j\right)
\big((1-a) -a(z_i+z_j+z_iz_j)\big)
\bigg)\\
&\kern2cm
\times
\bigg(
\prod _{0\le i<j\le n-1} ^{}
\left(\frac {1+z_j} {z_j}-\frac {1+z_i} {z_i}\right)\\
&\kern-13pt
=\frac {2} {n!}
\CT_{\mathbf z}
\bigg(\prod _{j=0} ^{n-1}\frac {(1+z_j)^{x-n}} {z_j^{n-1}}\bigg)
\bigg(
\prod _{0\le i<j\le n-1} ^{}
\left(z_i-z_j\right)^2
\big((1-a) -a(z_i+z_j+z_iz_j)\big)
\bigg).
\end{align*}
Now, the square of the Vandermonde product, $\prod _{0\le i<j\le n-1} ^{}
\left(z_i-z_j\right)^2$, is a homogeneous polynomial of degree $n(n-1)$.
Moreover, it is not very difficult to see that the coefficient of
$(z_0z_1\cdots z_{n-1})^{n-1}$ in it equals~$(-1)^{\binom n2}n!$. 
This implies that
\begin{align*}
\det_{0\le i,j\le n-1}&\left(a^i\binom {x+i+j-1}j+\binom {x-i+j-1}j\right)\\
&=2(-1)^{\binom n2}
\CT_{\mathbf z}
\bigg(\prod _{j=0} ^{n-1}(1+z_j)^{x-n}\bigg)
\bigg(
\prod _{0\le i<j\le n-1} ^{}
\big((1-a) -a(z_i+z_j+z_iz_j)\big)
\bigg)\\
&=2(a-1)^{\binom n2},
\end{align*}
as desired.
\end{proof}

\section{Proof of Conjecture \ref{conj:DiFran}}
\label{sec:conj1}

Here we prove Conjecture~\ref{conj:DiFran} using the holonomic Ansatz.

\begin{theorem} \label{thm:DiFran}
For all positive integers $n$, we have
\begin{equation} \label{eq:DiFran1} 
  \det_{0\leq i,j\le n-1} \left(2^i \binom{i+2j+1}{2j+1} - \binom{i-1}{2j+1}\right) =
  2\prod_{i=1}^n \frac{2^{i-1} \, (4i-2)!}{(n+2i-1)!},
\end{equation}
where the binomial coefficient is defined as in~\eqref{eq:bin}.
\end{theorem}

\begin{proof}
  We apply the holonomic Ansatz, described in Section~\ref{sec:holonom}.
  Computational details can be found in the accompanying electronic
  material~\cite{EM}.

  We are able to guess three recurrence relations for the quantities $c_{n,j}$, as
  defined in~\eqref{cnj}, whose shape suggests that they indeed form a holonomic
  sequence. The recurrences are too big to be displayed here (they would
  require approximately one page), so we give only their supports instead:
  \[
    \{c_{n,j+2},c_{n+1,j},c_{n,j+1},c_{n,j}\}, \quad
    \{c_{n+1,j+1},c_{n+1,j},c_{n,j+1},c_{n,j}\}, \quad
    \{c_{n+2,j},c_{n+1,j},c_{n,j+1},c_{n,j}\}.
  \]
  When translated into operator notation --- $\Sn$ denoting the forward shift
  operator $n\mapsto n+1$ --- their supports can be written more
  compactly as
  \begin{equation}\label{eq:supp}
    \{\Sj^2, \Sn, \Sj, 1\}, \quad
    \{\Sn \Sj, \Sn, \Sj, 1\}, \quad
    \{\Sn^2, \Sn, \Sj, 1\}.
  \end{equation}
  The corresponding operators form a (left) Gr\"obner basis, which is a useful
  property, as we will see later. During the guessing process, we have taken
  care that the final operators will have this property. Also for later use,
  we denote by~$\mathfrak{I}$ the annihilator ideal they generate.

  We want to show that the guessed recurrences (represented by~$\mathfrak{I}$)
  produce the correct values of $c_{n,j}$ for all~$j$ with
  $0\leq j<n$. For this
  purpose, we introduce another sequence $\tilde{c}_{n,j}$ that is defined
  via~$\mathfrak{I}$, and we show that it actually agrees with the
  sequence~$c_{n,j}$. The latter will be done by verifying that~\eqref{H1}
  and~\eqref{H2} hold when $c_{n,j}$ is replaced by~$\tilde{c}_{n,j}$.

  From the leading monomials $\Sj^2,\, \Sn\Sj,\, \Sn^2$ in~\eqref{eq:supp} one
  can deduce, using the theory of Gr\"obner bases, that the holonomic rank
  of~$\mathfrak{I}$ is three. Stated differently, the three irreducible
  monomials $1,\,\Sj,\,\Sn$ necessitate to specify initial values
  $\tilde{c}_{1,0},\,\tilde{c}_{1,1},\,\tilde{c}_{2,0}$ in
  order to fix a particular solution of the annihilator~$\mathfrak{I}$. Hence,
  we define $\tilde{c}_{n,j}$ to be the unique solution of~$\mathfrak{I}$
  whose three initial values agree with~$c_{n,j}$.
  
  From this definition of $\tilde{c}_{n,j}$ one can derive algorithmically a
  (univariate) recurrence for the almost-diagonal sequence
  $\tilde{c}_{n,n-1}$. This recurrence has order~$3$, which is equal to the
  holonomic rank of~$\mathfrak{I}$, as expected.  The corresponding operator
  has the right factor $\Sn-1$, and more precisely, it can be written in the form
  \begin{multline*}
     \bigl(9 (n+4) (2 n+5) (3 n+2) (3 n+4) (3 n+5) (3 n+7) p_1(n) \Sn^2 \\
     + 12 (3 n+2) (3 n+4) (4 n+3) (4 n+5) p_2(n) \Sn \\
     - 16 n (2 n+1) (4 n-1) (4 n+1) (4 n+3) (4 n+5) p_1(n+1)\bigr) \cdot (\Sn - 1),
  \end{multline*}
  where $p_1(n)$ and $p_2(n)$ are irreducible polynomials of degree~$9$
  and~$11$, respectively. It follows that any constant sequence is a solution
  of this recurrence.  Together with the initial conditions
  $\tilde{c}_{1,0}=\tilde{c}_{2,1}=\tilde{c}_{3,2}=1$, which are easy to
  check, this proves that $\tilde{c}_{n,n-1}=1$ holds for all $n\geq1$.

  The proof of the summation identity~\eqref{H2} is achieved by the method
  of creative telescoping, which delivers a set of recurrence equations
  (in~$n$ and~$i$) that are satisfied by the sum. For reasons of efficiency,
  we split the sum in \eqref{H2} into two sums as follows:
  \[
    \sum_{j=0}^{n-1} a_{i,j} \tilde{c}_{n,j} =
    \sum_{j=0}^{n-1} 2^i \binom{i+2j+1}{2j+1} \tilde{c}_{n,j}
    - \sum_{j=0}^{n-1} \binom{i-1}{2j+1} \tilde{c}_{n,j}.
  \]
  For each of the two sums, we obtain an annihilator ideal that is generated by
  four operators whose supports are as follows:
  \begin{align*}
    & \{\Si^3, \Sn^2, \Sn\Si , \Si^2, \Sn, \Si, 1\}, \quad
    \{\Si^2 \Sn, \Sn^2, \Sn\Si, \Si^2, \Sn, \Si, 1\}, \\
    & \{\Si \Sn^2, \Sn^2, \Sn\Si, \Si^2, \Sn, \Si, 1\}, \quad
    \{\Sn^3, \Sn^2, \Sn\Si, \Si^2, \Sn, \Si, 1\}.
  \end{align*}
  Actually, the two sums are annihilated by the very same operators, hence
  these operators constitute an annihilator for the left-hand side of~\eqref{H2}.
  The leading terms of the operators have the form:
  \begin{align*}
    & 12 (i-1) i (i+1) (3 n+1) (3 n+4) (4 n-1) (4 n+1) (i-n+3) (i-n+4) q_1(i,n) \Si^3, \\
    & {-9} i (3 n-1) (3 n+1) (3 n+4) q_2(i,n) \Sn\Si^2, \\
    & {-18} (i-1) i (n+1) (2 n+3) (3 n-1) (3 n+1)^2 (3 n+2) (3 n+4) (i+2 n+5) q_3(i,n) \Sn^2\Si, \\
    & {-54} (n+1) (n+2) (2 n+3) (2 n+5) (3 n-1) (3 n+1)^2 (i-2 n-6) (i-2 n-5) q_4(i,n) \Sn^3,
  \end{align*}
  where $q_1,q_2,q_3,q_4$ are (not necessarily irreducible) polynomials in~$n$ and~$i$.
  It remains to check a finite set of initial values. The shape of this set is
  determined by the support displayed above, by the condition $i<n-1$, and
  by the zeros of the leading coefficients of the operators. More precisely
  we have to verify that $\sum_{j=0}^{n-1} a_{i,j} \tilde{c}_{n,j} = 0$ for
  \begin{align*}
    (i,n) \in \{&(0, 2), (0, 3), (0, 4), (0, 5), (0, 6), (1, 3), (1, 4), (1, 5), (2, 4), \\
    & (1, 6), (1, 7), 
    (2, 5), (2, 6), 
    (2, 7), (2, 8), 
    (3, 5), (4, 6)\} 
  \end{align*}
  (where the points in the first line are determined by the support, and the
  second line is determined by the zeros of the leading coefficients).
  This verification is successful, and hence it follows that $\tilde{c}_{n,j}=c_{n,j}$
  for all~$j$ with $0\leq j<n$,
  which allows us to use $\mathfrak{I}$ as a holonomic
  definition of~$c_{n,j}$.

  In order to derive a recurrence for the left-hand side of~\eqref{H3} we
  split the sum into two sums, as before:
  \[
    \sum_{j=0}^{n-1} a_{n,j} c_{n,j} =
    \sum_{j=0}^{n-1} 2^n \binom{n+2j+1}{2j+1} c_{n,j}
    - \sum_{j=0}^{n-1} \binom{n-1}{2j+1} c_{n,j}.
  \]
  Then we compute, for each of the two
  sums, a recurrence by creative telescoping. In both cases, the output is a
  recurrence of order~$6$ with polynomial coefficients of degree
  approximately~$52$. Actually one finds that both sums satisfy the same
  order-$6$ recurrence, and hence so does their sum.  One now has to verify
  that $b_n/b_{n-1}$ satisfies this order-$6$ recurrence, where $b_n$ denotes
  the right-hand side of \eqref{eq:DiFran1}. We have
  \[
    \frac{b_n}{b_{n-1}} =
    \frac{(4n-2)!}{(3n-1)! \, \bigl(\frac{n+1}{2}\bigr)_{n-1}}.
  \]
  Note that this expression is hypergeometric in $n/2$ and hence satisfies a
  second-order recurrence whose operator has support $\{\Sn^2,1\}$.
  Right-dividing the operator of the order-$6$ recurrence, call it~$P$, by
  this second-order operator yields~$0$, hence $P$ annihilates $b_n/b_{n-1}$.
  The leading term of the operator~$P$ is
  \begin{multline*}
     4374 (n+7) (2 n+9) (2 n+11) (3 n-1) (3 n+1) (3 n+2) (3 n+4) (3 n+5) (3 n+7) \\
     \times (3 n+8) (3 n+10) (3 n+11) (3 n+13)^2 (3 n+14) (3 n+16) (3 n+17) p(n) \Sn^6,
  \end{multline*}
  where $p(n)$ is an irreducible polynomial of degree~$35$. Obviously this
  leading coefficient does not vanish for any positive integer~$n$, hence
  it suffices to verify
  \[
    \frac{\det_{0\leq i,j\le n-1}(a_{i,j})}{\det_{0\leq i,j\le n-2}(a_{i,j})} = \frac{b_n}{b_{n-1}}
  \]
  for $n=2,\dots,7$. On both sides, one calculates the values $4$, $15$, $832/15$,
  $204$, $9728/13$, $16445/6$, respectively.  By virtue of the recurrence~$P$, the
  asserted identity~\eqref{eq:DiFran1} holds for all integers $n\geq1$.
\end{proof}

\section{Proof of Conjecture \ref{conj:CK}}
\label{sec:conj2}

In this section, we present our proofs of \eqref{eq:CK1}
and~\eqref{eq:CK2}. As was mentioned in the introduction, it turned
out that the capacity of today's computers is not sufficient for the
holonomic Ansatz to produce proofs of theses two identities, although
it very likely applies. Instead,
the starting point for our proofs is
determinant evaluations that have been established in~\cite{CoHuKr23}.
In their statements, there appears the {\it Delannoy number} $D(i,j)$,
which by definition is the number of paths from $(0,0)$ to $(i,j)$
consisting of right-steps $(1,0)$, up-steps $(0,1)$, and diagonal
steps $(1,1)$. Their generating function is given by (cf.\
\cite[Ex.~21 in Ch.~I]{ComtAA})
\begin{equation} \label{eq:f-coef}
D(i,j)=\coef{u^iv^j}\frac {1} {1-u-v-uv},
\end{equation}
where $\coef{u^iv^j}g(u,v)$ denotes the coefficient of $u^iv^j$ in
the formal power series (in the variables~$u$ and~$v$) $g(u,v)$.
The following result is \cite[Th.~5.1, in combination with
Eqs.~(4.3)--(4.5) and
paragraph above and including Eq.~(4.6)]{CoHuKr23}

\begin{theorem}
For all positive integers $k$ and $n$, we have
\begin{multline} \label{eq:D1}
D_1(k;n):=
\det_{1 \leq i,j \leq k} \left( D(2j-i, i+n-k-1) \right)\\
=        \left.\prod_{i\ge 0} \left( \prod_{s=-2k+4i+1}^{-k+2i}(2n+s)\prod_{s=k-2i}^{2k-4i-2}(2n+s) \right) \middle/ \prod_{i=1}^{k-1} (2i+1)^{k-i} \right.
        ,
\end{multline}
\end{theorem}

We are now prepared for the proof of \eqref{eq:CK1}, which we
restate below with a modified, but equivalent, right-hand side.

\begin{theorem} \label{thm:D3}
For all positive integers $n$, we have
\begin{align}
D_2(n;x) :&{=}
\det_{0\le i,j\le n-1}\left(2^i\binom {x+i+2j+1}{2j+1}+\binom {x-i+2j+1}{2j+1}\right) \notag
\\
&{=}\
2^{\binom n2+1}\prod _{i=0} ^{n-1}\frac {i!} {(2i+1)!}
\prod _{i=0} ^{\fl{n/2}}(x+4i+1)_{n-2i}
\prod _{i=0} ^{\fl{(n-1)/2}}(x-2i+3n)_{n-2i-1}, \notag
\\
&{=}\ 2\prod_{i=1}^n \frac{2^{2 i-2} \, \Gamma\bigl(i\bigr) \, \Gamma\bigl(2i+x\bigr) \,
  \Gamma\bigl(4i+x-1\bigr) \, \Gamma\bigl(\frac{3 i+x-2}{2}\bigr)}{%
  \Gamma\bigl(2i\bigr) \, \Gamma\bigl(3i+x\bigr) \, \Gamma\bigl(3i+x-2\bigr) \,
  \Gamma\bigl(\frac{i+x}{2}\bigr)},
\label{eq:D3}
\end{align}
where the binomial coefficients have to be interpreted
according to~\eqref{eq:bin}.
\end{theorem}

The theorem will, up to some routine manipulations, immediately follow
from the relation below.

\begin{lemma} \label{lem:CK1}
For all positive integers $k$, we have
\begin{equation} \label{eq:D2=D3} 
D_1(k;y+k)=\tfrac {1} {2}D_2(k;2y).
\end{equation}
\end{lemma}

\begin{proof}
We follow --- and extend --- Di~Francesco's arguments in
\cite[Proofs of Ths~3.3, 4.3, and~8.2]{DiFran21}. His idea is to work
with determinants of the form  $\det A(n)$ where
$A(n)=(a_{i,j})_{0\le i,j\le n-1}$, with
the entries $a_{i,j}$ given by a two-variable generating function,
$$
a(u,v)=\sum_{i,j\ge0}a_{i,j}u^iv^j.
$$
The determinant will be unchanged if the matrix is multiplied (from
the right or from the left) by a triangular matrix with $1$s on the
diagonal. It is easy to see that multiplication of $a(u,v)$ by a power
series in~$u$ or by a power series in~$v$ with constant coefficient~1
will result in the
multiplication of $A(n)$ by such a triangular matrix, and thus the
determinant of the new matrix is still the same. The same property
holds if in $a(u,v)$ we replace $u$ by a power series in~$u$ with zero
constant coefficient and coefficient of~$u$ equal to~1. Di~Francesco
argues with the help of complex integrals, but this is not necessary.

\medskip
We start with expressing $D_1(n;k)$ in the above
form. By shifting the row and column indices~$i$ and~$j$ by~1, we have
$$
D_1(k;n)=
\det_{0 \leq i,j \leq k-1} \left( D(2j-i+1, i+n-k) \right).
$$
By~\eqref{eq:f-coef} (with the roles of $u$ and $v$ interchanged), we have
$$
D(2j-i+1, n - k+i)=\coef{u^{n-k+i}v^{2j-i+1}}\frac {1} {1-u-v-uv}.
$$
By replacing $u$ by $uv$, we see that
$$
D(2j-i+1, n - k+i)=\coef{u^{n-k+i}v^{n-k+2j+1}}\frac {1} {1-v-uv-uv^2}.
$$
From here on, we write $N$ for $n-k$ for short, so that
$$
D(2j-i+1, N+i)=\coef{u^{N+i}v^{N+2j+1}}\frac {1} {1-v-uv-uv^2}.
$$
If we denote the coefficient of $u^iv^j$ in $1/(1-v-uv-uv^2)$ by
$\alpha_{i,j}$, then
\begin{align*}
\sum_{i,j\ge0}\alpha_{i,j}u^iv^j&=\frac {1} {1-v-uv-uv^2}\\
&=\sum_{s\ge0}\left(\frac {uv(1+v)} {1-v}\right)^s\frac {1} {1-v}.
\end{align*}
Consequently,
\begin{align*}
\sum_{i,j\ge0}\alpha_{i+N,j+N}u^{i}v^{j}&=
(uv)^{-N}\sum_{s\ge N}\left(\frac {uv(1+v)} {1-v}\right)^s\frac {1} {1-v}\\
&=\left(\frac {1+v} {1-v}\right)^N \frac {1} {1-v-uv-uv^2}.
\end{align*}
We have shown that
$$
D(2j-i+1, N+i)=\coef{u^{i}v^{2j+1}}\left(\frac {1+v} {1-v}\right)^N
\frac {1} {1-v-uv-uv^2}.
$$
By $2$-section (in~$v$) of the series on the right-hand side, we
finally obtain
\begin{align} 
\notag
  D(2j-i+1, N+i)&=\coef{u^{i}v^{j}}\frac {1} {2\vv}
\left(\left(\frac {1+\vv} {1-\vv}\right)^N\frac {1} {1-\vv-u\vv-uv}\right.\\
\notag
&\kern3cm\left.
-\left(\frac {1-\vv} {1+\vv}\right)^N\frac {1} {1+\vv+u\vv-uv}\right)\\
\notag
&\kern-3cm
=\coef{u^{i}v^{j}}\frac {1} {2\vv(1 - v - 4 u v - u^2 v + u^2 v^2)}\\
&\kern-2.5cm
\cdot
\left(\left(\frac {1+\vv} {1-\vv}\right)^N\big(1-uv+\vv(1+u\big)
-\left(\frac {1-\vv} {1+\vv}\right)^N\big(1-uv-\vv(1+u\big)\right).
\label{eq:coef1}
\end{align}
We have reached our first intermediate goal to express the determinant
$D_1(k;n)$ in the form $\det B(n)$,  where $B(n)=(b_{i,j})_{0\le i,j\le n-1}$
with $b(u,v)=\sum_{i,j\ge0}b_{i,j}u^iv^j$ the double series on the
right-hand side of~\eqref{eq:coef1}. (Recall that $N=n-k$.)

Now we transform our determinant by multiplying $b(u,v)$ by
$(1-v)^{-N}$. (The latter is indeed a power series in~$v$ with
constant coefficient equal to~1.\footnote{This step was not necessary
  in~\cite{DiFran21} since there $N=n-k=0$.}) Thus we see that
$D_1(k;n)=\det C(n)$,  where $C(n)=(c_{i,j})_{0\le i,j\le n-1}$ with
\begin{multline*}
c(u,v)=\sum_{i,j\ge0}c_{i,j}u^iv^j=\frac {1} {2\vv(1 - v - 4 u v - u^2 v + u^2 v^2)}\\
\times
\left((1-\vv)^{-2N}\big(1-uv+\vv(1+u)\big)
-(1+\vv)^{-2N}\big(1-uv-\vv(1+u\big)\right).
\end{multline*}
We transform this series (and thus the corresponding matrix) by
performing the substitution $u\mapsto\frac {u} {(1-u)(1-2u)}$,
followed by multiplication by $\frac {1-2u^2} {(1-u)(1-2u)}$.%
\footnote{Di Francesco does this transformation in two steps.
First, he does the substitution $u\mapsto u\frac {1+u} {1-u}$ and he
multiplies the resulting generating function by $\frac {1+2u-u^2}
{1-u}$. (At this point, he has shown the equality of the domino
tilings partition function with the 20-vertex partition function.)
Subsequently, he does the substitution $u\mapsto \frac {u} {1-u}$,
which leads him to the determinant in~\eqref{eq:D3} with $x=0$.
(A subtlety is that he does not arrive exactly at the determinant
in~\eqref{eq:D3} but rather at the matrix that arises from ours by
dividing all entries in the $0$-th row  by~2 because his binomial
coefficient $\binom {i-1}{2j+1}$ must be interpreted as~0 for $i=0$ ---
as opposed to our convention concerning the binomial coefficient.)
We have combined these two steps here into one.}
As a result, we obtain that 
$D_1(k;n)=\det D(n)$,  where $D(n)=(d_{i,j})_{0\le i,j\le n-1}$ with
\begin{multline} \label{eq:d(u,v)}
d(u,v)=\sum_{i,j\ge0}d_{i,j}u^iv^j=\frac {1-2u^2}
  {2\vv\big((1-2u)^2-v\big)\big((1-u)^2-u^2v\big)}\\
\times
\left((1-\vv)^{-2N}\left((1-u)(1-2u)-uv+\vv
     (1-2u+2u^2)\right)\right.\\
\left.
  -(1+\vv)^{-2N}\left((1-u)(1-2u)-uv-\vv
     (1-2u+2u^2)\right)\right).
\end{multline}

\medskip
We turn our attention to the determinant in \eqref{eq:D3}.
We have
\begin{equation} \label{eq:binom} 
\sum_{j\ge0}\binom {x+i+j}jv^j=(1-v)^{-x-i-1}
\quad \text{and}\quad 
\sum_{j\ge0}\binom {x-i+j}jv^j=(1-v)^{-x+i-1},
\end{equation}
and therefore, again by a 2-section,
\begin{equation} \label{eq:GF1} 
\sum_{j\ge0}\binom {x+i+2j+1}{2j+1}v^{j}
=\frac {1} {2\vv}\left((1-\vv)^{-x-i-1}-(1+\vv)^{-x-i-1}\right)
\end{equation}
and
\begin{equation} \label{eq:GF2} 
\sum_{j\ge0}\binom {x-i+2j+1}{2j+1}v^{j}
=\frac {1} {2\vv}\left((1-\vv)^{-x+i-1}-(1+\vv)^{-x+i-1}\right).
\end{equation}
Consequently, we have 
$D_2(n;x)=\det E(n)$,  where $E(n)=(e_{i,j})_{0\le i,j\le n-1}$ with
\begin{align} \notag
e(u,v)&=\sum_{i,j\ge0}e_{i,j}u^iv^j=\frac {1} {2\vv}
\left((1-\vv)^{-x-1}\frac {1} {1-\frac {2u} {1-\vv}}
-(1+\vv)^{-x-1}\frac {1} {1-\frac {2u} {1+\vv}}\right.\\
\notag
&\kern2cm\left.
+(1-\vv)^{-x-1}\frac {1} {1-u(1-\vv)}
-(1+\vv)^{-x-1}\frac {1} {1-u(1+\vv)}
\right)\\
\notag
&=\frac {1} {2\vv}
\left((1-\vv)^{-x}\frac {1-2u+\vv} {(1-2u)^2-v}
-(1+\vv)^{-x}\frac {1-2u-\vv} {(1-2u)^2-v}\right.\\
&\kern2cm\left.
+(1-\vv)^{-x-1}\frac {1-u-u\vv} {(1-u)^2-u^2v}
-(1+\vv)^{-x-1}\frac {1-u+u\vv} {(1-u)^2-u^2v}
\right).
\label{eq:e(u,v)}
\end{align}

In order to explain the factor $\frac {1} {2}$ in~\eqref{eq:D2=D3}, we
now want to divide all entries in row~0 of the current matrix $E(n)$
by~2. In terms of generating functions, we achieve this by
subtracting {\it half\/} of the generating function for the entries in row~0,
$$
\frac {1} {2}\sum_{j\ge0}e_{0,j}v^j=
\frac {1} {2\vv}
\left((1-\vv)^{-x-1}
-(1+\vv)^{-x-1}\right),
$$
from $e(u,v)$. We are led to the conclusion that
$\frac {1} {2}D_2(n;x)=\det F(n)$,  where $F(n)=(f_{i,j})_{0\le i,j\le n-1}$ with
\begin{align*} 
f(u,v)&=\sum_{i,j\ge0}f_{i,j}u^iv^j=\frac {1} {2\vv}
\left((1-\vv)^{-x}\frac {1-2u+\vv} {(1-2u)^2-v}
-(1+\vv)^{-x}\frac {1-2u-\vv} {(1-2u)^2-v}\right.\\
&\kern2cm
+(1-\vv)^{-x-1}\frac {u(1-u-u\vv)(1-\vv)} {(1-u)^2-u^2v}\\
&\kern4cm\left.
-(1+\vv)^{-x-1}\frac {u(1-u+u\vv)(1+\vv)} {(1-u)^2-u^2v}
\right)\\
&=\frac {1} {2\vv}
\left((1-\vv)^{-x}\frac {1-2u+\vv} {(1-2u)^2-v}
-(1+\vv)^{-x}\frac {1-2u-\vv} {(1-2u)^2-v}\right.\\
&\kern2cm\left.
+(1-\vv)^{-x}\frac {u(1-u-u\vv)} {(1-u)^2-u^2v}
-(1+\vv)^{-x}\frac {u(1-u+u\vv)} {(1-u)^2-u^2v}
\right).
\end{align*}
One can now readily verify that $f(u,v)$ with $x=2N$ is equal to
$d(u,v)$ as given in~\eqref{eq:d(u,v)}. In view of $N=n-k$, this
establishes the relationship~\eqref{eq:D2=D3}. This completes the
proof of the theorem.
\end{proof}

Now we prove \eqref{eq:CK2}, restated again with a modified, but
equivalent, right-hand side.

\begin{theorem} \label{thm:D4}
For all positive integers $n$, we have
\begin{align} 
D_3(n;x) :&{=}
\det_{0\le i,j\le n-1}\left(2^i\binom {x+i+2j}{2j}+\binom {x-i+2j}{2j}\right)
\notag
\\
&{=}\ 
2^{\binom n2+1}\prod _{i=0} ^{n-1}\frac {i!} {(2i)!}
\prod _{i=0} ^{\fl{(n-1)/2}}(x+4i+3)_{n-2i-1}
\prod _{i=0} ^{\fl{(n-2)/2}}(x-2i+3n-1)_{n-2i-2}, \notag
\\
&{=}\ 2\prod_{i=1}^n \frac{2^{2i-2} \, \Gamma\bigl(i\bigr) \, \Gamma\bigl(2i+x\bigr) \,
  \Gamma\bigl(4i+x-3\bigr) \, \Gamma\bigl(\frac{3i+x-1}{2}\bigr)}{\Gamma\bigl(2i-1\bigr) \,
  \Gamma\bigl(3i+x-1\bigr) \, \Gamma\bigl(3i+x-2\bigr) \, \Gamma\bigl(\frac{i+x+1}{2}\bigr)},
\label{eq:D4}
\end{align}
where the binomial coefficients have to be interpreted
according to~\eqref{eq:bin}.
\end{theorem}

Again, the theorem will, up to some routine manipulations, immediately follow
from a relation between the above determinant and the earlier
determinant $D_1(k;n)$ defined in~\eqref{eq:D1}.

\begin{lemma} \label{lem:CK2}
For all positive integers $k$, we have
\begin{equation} \label{eq:D2=D4} 
D_1(k-1;y+k)=\tfrac {1} {2}D_3(k;2y).
\end{equation}
\end{lemma}

\begin{proof}
Here we start with $D_3(n;x)$. From 2-section of the binomial
series in~\eqref{eq:binom}, we obtain that
$D_3(n;x)=\det H(n)$,  where $H(n)=(h_{i,j})_{0\le i,j\le n-1}$ with
\begin{align*} 
h(u,v)&=\sum_{i,j\ge0}h_{i,j}u^iv^j=\frac {1} {2}
\left((1-\vv)^{-x-1}\frac {1} {1-\frac {2u} {1-\vv}}
+(1+\vv)^{-x-1}\frac {1} {1-\frac {2u} {1+\vv}}\right.\\
&\kern2cm\left.
+(1-\vv)^{-x-1}\frac {1} {1-u(1-\vv)}
+(1+\vv)^{-x-1}\frac {1} {1-u(1+\vv)}
\right)\\
&=\frac {1} {2}
\left((1-\vv)^{-x}\frac {1-2u+\vv} {(1-2u)^2-v}
+(1+\vv)^{-x}\frac {1-2u-\vv} {(1-2u)^2-v}\right.\\
&\kern2cm\left.
+(1-\vv)^{-x-1}\frac {1-u-u\vv} {(1-u)^2-u^2v}
+(1+\vv)^{-x-1}\frac {1-u+u\vv} {(1-u)^2-u^2v}
\right).
\end{align*}
It should be noted that the only differences with \eqref{eq:e(u,v)}
are that, here, the prefactor is $\frac {1} {2}$ instead of $\frac {1}
{2\vv}$, and that there are plus-signs in front of the terms involving
$(1+\vv)^{-x}$. Hence, if we proceed from here as in the proof of
Lemma~\ref{lem:CK1} --- that is, we divide the 0-th row of $H(n)$
by~2, and then do the transformations described in the proof of
Lemma~\ref{lem:CK1} ``in reverse" --- then we obtain
$$
\frac {1} {2}D_3(k;2y)=\det_{0\le i,j\le k-1}
\big(D(2j-i,y+i)\big).
$$
Here we see that
all entries in column~$0$ of the last matrix are zero except for the
entry in row~$0$ which is equal to $D(0,y)=1$. By expanding the
determinant of this matrix along the first column, we see that
$$
\frac {1} {2}D_3(k;2y)=\det_{0\le i,j\le k-2}
\big(D(2j-i+1,y+i+1)\big)=D_1(k-1;y+k).
$$
This is exactly \eqref{eq:D2=D4}.
\end{proof}

\section{Variations on the theme, I}
\label{sec:var1}

There exist numerous variations of Theorem \ref{thm:DiFran} in which
the exponent in the exponential~$2^i$ is shifted. In this section, we
report our corresponding findings. Let
\[
  D_{\alpha,\beta,\gamma,\delta}(n) :=
  \det_{0\leq i,j\le n-1}\left(
  2^{i+\beta}\binom{i+2j+\gamma}{2j+\alpha}
  +\binom{-i+2j+\delta}{2j+\alpha}\right).
\]
Note that, in this notation, the determinants from
Theorems~\ref{thm:D3} and~\ref{thm:D4} read
\begin{align}
  D_2(n;x) &= D_{1,0,x+1,x+1}(n),\\
  D_3(n;x) &= D_{0,0,x,x}(n),
\end{align}
respectively.

In an automated search in the parameter space $-6\leq\alpha,\beta\leq9$ and
$-9\leq\gamma,\delta\leq9$, we have identified $26$ cases of determinants that
factor completely, and which are not special instances of $D_2(n;x)$ or
$D_3(n;x)$. All of these 26 cases can be proven automatically by the holonomic
Ansatz, but some of them can also easily be related to each other.

\begin{theorem}\label{thm:det22}
The following determinant evaluations hold for all $n\geq1$:
\begin{align}
  D_{-2, 0, -1, -1}(n) &= -2\prod_{i=2}^{n} \frac{8(2i-3)(2i-1)\,\Gamma\bigl(4i-5\bigr)\,
    \Gamma\bigl(\frac{i+1}{2}\bigr)}{i\,\Gamma\bigl(3i-2\bigr)\,
    \Gamma\bigl(\frac{3i-3}{2}\bigr)}, \label{eq:det22a} 
  \\
  D_{0, 2, 3, -1}(n) &= \prod_{i=1}^{n} \frac{3(2i-1)\,\Gamma\bigl(4i+3\bigr)\,
    \Gamma\bigl(\frac{i+1}{2}\bigr)}{4(i+2)\,\Gamma\bigl(3i+1\bigr)\,
    \Gamma\bigl(\frac{3i+5}{2}\bigr)}, \label{eq:det22b} 
  \\
  D_{1, 1, 0, -2}(n) &= -2\prod_{i=1}^{n} \frac{(2i-1)\,\Gamma\bigl(4i-3\bigr)\,
    \Gamma\bigl(\frac{i}{2}\bigr)}{2\,\Gamma\bigl(3i-2\bigr)\,
    \Gamma\bigl(\frac{3i}{2}\bigr)} ,\label{eq:det22c} 
  \\
  D_{1, 1, 1, -1}(n) &= \prod_{i=1}^{n} \frac{\Gamma\bigl(4i-1\bigr)\,
    \Gamma\bigl(\frac{i+1}{2}\bigr)}{\,\Gamma\bigl(3i\bigr)\,
    \Gamma\bigl(\frac{3i-1}{2}\bigr)}, \label{eq:det22d} 
  \\
  D_{2, 1, 2, 0}(n) &= \prod_{i=1}^{n} \frac{\Gamma\bigl(4i\bigr)\,
    \Gamma\bigl(\frac{i+2}{2}\bigr)}{\,\Gamma\bigl(3i\bigr)\,
    \Gamma\bigl(\frac{3i+2}{2}\bigr)}, \label{eq:det22e} 
  \\
  D_{0, 1, 1, -1}(n) &= 3\prod_{i=2}^{n} \frac{\Gamma\bigl(4i\bigr)\,
    \Gamma\bigl(\frac{i-1}{2}\bigr)}{\,\Gamma\bigl(3i+1\bigr)\,
    \Gamma\bigl(\frac{3i-3}{2}\bigr)}. \label{eq:det22f} 
\end{align}
Moreover, some related determinants can be expressed in terms of these;
the following identities hold (at least) for all $n\geq4$:
\begin{align}
  D_{2, 1, 2, 0}(n) &= 
  \frac18 \, D_{1, 1, -1, -3}(n+1) = 
  \frac{1}{40} \, D_{0, 1, -4, -6}(n+2) = 
  -\frac{1}{24576} \, D_{1, 2, -4, -8}(n+2), 
  \label{eq:det22g}
  \\
  D_{1, 1, 1, -1}(n) &= 
  D_{2, 1, 1, -1}(n) = 
  \frac13 \, D_{0, 1, -2, -4}(n+1) = 
  -\frac{1}{32} \, D_{1, 1, -2, -4}(n+1) \notag \\ &= 
  -\frac{1}{224} \, D_{1, 2, -2, -6}(n+1) = 
  -\frac{1}{168} \, D_{0, 1, -5, -7}(n+2) \notag \\ &= 
  -\frac{1}{3696} \, D_{0, 2, -5, -9}(n+2) = 
  -\frac{1}{337920} \, D_{1, 2, -5, -9}(n+2), 
  \label{eq:det22h}
  \\
  D_{1, 1, 0, -2}(n) &= \frac15 \, D_{0, 1, -3, -5}(n+1) = \frac{1}{1008} \, D_{1, 2, -3, -7}(n+1),
  \label{eq:det22i}
  \\
  D_{-2, 1, 0, -2}(n) &= D_{0, 2, 3, -1}(n-1), 
  \label{eq:det22j}
  \\
  D_{2, 1, 1, -1}(n) &= D_{4, 2, 4, 0}(n-1), 
  \label{eq:det22k}
  \\
  D_{1, 1, -2, -4}(n) &= -\frac{16}{5} \, D_{3, 2, 1, -3}(n-1) = 
  \frac{64}{3} \, D_{5, 3, 4, -2}(n-2) = -128 \, D_{7, 4, 7, -1}(n-3), 
  \label{eq:det22l}
  \\
  D_{1, 1, -1, -3}(n) &= -4 \, D_{3, 2, 2, -2}(n-1) = 16 \, D_{5, 3, 5, -1}(n-2), 
  \label{eq:det22m}
  \\
  D_{1, 1, 0, -2}(n) &= -2 \, D_{3, 2, 3, -1}(n-1).
  \label{eq:det22n}
\end{align}
\end{theorem}
\begin{proof}
Identities \eqref{eq:det22a}--\eqref{eq:det22i} can be proven by the holonomic
Ansatz. More precisely, we prove a closed-form evaluation for each of the
mentioned determinants, similar to those in
\eqref{eq:det22a}--\eqref{eq:det22f}, but find that some of these are related
to each other. In order to make these relations explicit, and in order to save
some space, we display in \eqref{eq:det22g}--\eqref{eq:det22i} only the
relations, not the closed forms themselves. The detailed proofs can be found
in the accompanying electronic material~\cite{EM}, some computational data
are given in Table~\ref{tab:comp}.

Identities \eqref{eq:det22j}--\eqref{eq:det22n} can easily be established by
exploiting the structure of the corresponding matrices: the matrices of the
determinants on the right-hand sides take the block form
$\begin{psmallmatrix} C & 0 \\ \ast & A \end{psmallmatrix}$ or
$\begin{psmallmatrix} C & \ast \\ 0 & A \end{psmallmatrix}$
where in each case $C$ is a fixed matrix of dimension $1\times1$, or
$2\times2$, or $3\times3$, and where $A$ is the matrix of the
determinant on the corresponding right-hand side. The latter follows from the
fact that the transformation
$(\alpha,\beta,\gamma,\delta)\mapsto(\alpha+2,\beta+1,\gamma+3,\delta+1)$ is
equivalent to shifting $(i,j)\mapsto(i+1,j+1)$.
\end{proof}

By looking at \eqref{eq:det22g}--\eqref{eq:det22i} one is tempted to
prove these relations directly, without taking the detour via the
closed-form evaluations. We demonstrate with one example how this can work.
Let $L_n$ be the lower-triangular $(n\times n)$-matrix with
entries $2^{i-j+1}-1$ and $R_n$ be the
$(n\times n)$-matrix with $1$'s on the main
diagonal, $-1$'s on the upper diagonal, and $0$ elsewhere, i.e.,
\[
  L_n := \begin{pmatrix}
    1 & 0 & 0 & 0 & \cdots \\ 
    3 & 1 & 0 & 0 & \\
    7 & 3 & 1 & 0 & \\
    15 & 7 & 3 & 1 & \\
    \smash{\vdots} & & & & \smash{\ddots}
  \end{pmatrix}
  \qquad\text{and}\qquad
  R_n := \begin{pmatrix}
    1 & -1 & 0 & 0 & \cdots \\
    0 & 1 & -1 & 0 & \\
    0 & 0 & 1 & -1 & \\
    0 & 0 & 0 & 1 & \smash{\ddots} \\
    \smash{\vdots} & & & & \smash{\ddots}
  \end{pmatrix}.
\]
Then, for example, we claim that
\begin{equation}\label{eq:LR}
  L_n\cdot A_{2, 1, 2, 0}(n)\cdot R_n =
  \begin{pmatrix} 2 & 0 \\ \ast & A_{5, 3, 5, -1}(n-1) \end{pmatrix},
\end{equation}
where $A_{\alpha,\beta,\gamma,\delta}(n)$ denotes the matrix from the definition
of $D_{\alpha,\beta,\gamma,\delta}(n)$. Equation~\eqref{eq:LR}
immediately implies the identity
$D_{2, 1, 2, 0}(n) = 2 \, D_{5, 3, 5, -1}(n-1)$, which has already been stated
(implicitly; it is obtained by combining \eqref{eq:det22g}
with~\eqref{eq:det22m}). It remains to show~\eqref{eq:LR}, which boils
down to proving the binomial sum identity
\begin{align}
\notag
  \sum_{k=0}^i \biggl[ & (1 - 2^{i-k+1}) \left(\binom{-1+2j-k}{2j+2} - \binom{1+2j-k}{2j+4}\right) \\
\notag
    &\kern2cm
    + (2^{k+2} - 2^{i+3}) \left(\binom{3+2j+k}{2j+2} - \binom{5+2j+k}{2j+4}\right) \biggr] \\
  &\kern.5cm
  = \binom{-1-i+2j}{2j+5} + 2^{i+3} \binom{5+i+2j}{2j+5}.
\label{eq:sumid}
\end{align}
This can be achieved by observing that
$$
\binom {a+k}b=\binom {a+k+1}{b+1}-\binom {a+k}{b+1}
$$
and\footnote{We found this relation by means of Gosper's
algorithm~\cite{GospAB}, using the implementation~\cite{PaScAA}.}
$$
2^k\left(\binom {a+k}b-\binom {a+k+2}{b+2}\right)
=
-2^{k+1}\binom {a+k+1}{b+2}+2^k\binom {a+k}{b+2},
$$
so that all sums in \eqref{eq:sumid} are telescoping sums.
The identity can also be automatically proved by
Zeilberger's algorithm~\cite{Zeilberger90a,PetkovsekWilfZeilberger96}. 

\medskip
Similarly to \eqref{eq:det22g}--\eqref{eq:det22n} one can connect the
determinants from Theorem~\ref{thm:det22} to the determinants $D_2(n;x)$ and
$D_3(n;x)$, whose evaluations have already been proven in Theorems~\ref{thm:D3}
and~\ref{thm:D4}, respectively.
\begin{corollary}
  The following identities hold for all integers $n\geq2$:
  \begin{align}
    2 \, D_{1, 1, 1, -1}(n) &= D_3(n+1;-2) = D_3(n;1) = D_2(n;0) = D_2(n-1;3), \\
    2 \, D_{2, 1, 2, 0}(n) &= D_3(n+1;-1) = D_2(n;1), \\
    -D_{1, 1, 0, -2}(n) &= D_3(n;0) = D_2(n-1;2).
  \end{align}
\end{corollary}

\section{Variations on the theme, II}
\label{sec:var2}

In this section we present --- and prove --- several determinant evaluations
in which, compared with Theorem~\ref{thm:DiFran}, the power $2^i$ gets
replaced by~$3^i$, and the terms $2j$ in the binomials are replaced by~$3j$.
As it turns out, there are even more variations of Theorem~\ref{thm:DiFran}
associated with the modulus~3 if one also shifts the exponent in the
exponential~$3^i$. For brevity, let us denote
\[
  E_{\alpha,\beta,\gamma,\delta}(n) :=
  \det_{0\leq i,j\le n-1}\left(3^{i+\beta} \binom{i+3j+\gamma}{3j+\alpha}
  +\binom{-i+3j+\delta}{3j+\alpha}\right).
\]
An automated search in the parameter space
\begin{multline*}
   \bigl\{(\alpha,\beta,\gamma,\delta) : -6\leq\alpha,\beta\leq6
  \text{ and } {-8}\leq\gamma,\delta\leq8\bigr\} \\
  {} \cup {} 
  \bigl\{(\alpha,\beta,\gamma,\delta) : 6\leq\alpha\leq10
  \text{ and } 0\leq\beta\leq10 \text{ and } {-10}\leq\gamma,\delta\leq10\bigr\}
\end{multline*}
delivered $26$ cases of determinants that factor completely. All of these
$26$~cases can be proven automatically by the holonomic Ansatz (see the
accompanying electronic material~\cite{EM}), but some of them can also easily
be related to each other.

Finally, we have discovered three parametric families of determinant
evaluations of this kind, in addition to the other, (seemingly) sporadic ones.
The parametric families are presented in Conjecture~\ref{conj:det33x} below. Here,
it seems difficult to apply the holonomic Ansatz, but purely because of the
computational complexity that is added by the parameter~$x$.  We are
absolutely convinced that it should work in principle, since we observed that
it works for specific values of~$x$ without much difficulty.
We admit that we do not
know a different method that would work here.

\begin{theorem}\label{thm:det33}
The following determinant evaluations hold for all $n\geq1$:
\begin{align}
E_{-3, 0, -1, -1}(n) &= 2\prod_{i=2}^{n} \frac{2^{i+1}(2i-1)\,\Gamma\bigl(4i-5\bigr)\,\Gamma\bigl(\frac{i+2}{3}\bigr)}{i(i+1)\,\Gamma\bigl(3i-5\bigr)\,\Gamma\bigl(\frac{4i-1}{3}\bigr)}, 
\label{eq:det33a} \\
E_{-3, 1, 0, -2}(n) &= -2\prod_{i=2}^{n} \frac{2^{i+1}(2i-1)\,\Gamma\bigl(4i-4\bigr)\,\Gamma\bigl(\frac{i}{3}\bigr)}{i(i+1)^2\,\Gamma\bigl(3i-5\bigr)\,\Gamma\bigl(\frac{4i-3}{3}\bigr)}, 
\label{eq:det33b} \\
E_{0, 3, 5, -1}(n) &= \prod_{i=1}^{n} \frac{2^{i+1}(3i-2)(3i-1)\,\Gamma\bigl(4i+4\bigr)\,\Gamma\bigl(\frac{
i+2}{3}\bigr)}{(i+1)(i+2)(i+3)(i+4)\,\Gamma\bigl(3i+1\bigr)\,\Gamma\bigl(\frac{4i+5}{3}\bigr)}, 
\label{eq:det33c} \\
E_{0, 1, 1, -1}(n) &= \prod_{i=1}^{n} \frac{2^{i+1}\,\Gamma\bigl(4i-2\bigr)\,\Gamma\bigl(\frac{i+2}{3}\bigr)}{i\,\Gamma\bigl(3i-2\bigr)\,\Gamma\bigl(\frac{4i-1}{3}\bigr)}, 
\label{eq:det33d} \\
E_{1, 1, 2, 0}(n) &= \prod_{i=1}^{n} \frac{2^i\,\Gamma\bigl(4i\bigr)\,\Gamma\bigl(\frac{i+1}{3}\bigr)}{3i\,\Gamma\bigl(3i-1\bigr)\,\Gamma\bigl(\frac{4i+1}{3}\bigr)}, 
\label{eq:det33e} \\
E_{3, 2, 3, -1}(n) &= \prod_{i=1}^{n} \frac{2^i\,\Gamma\bigl(4i+1\bigr)\,\Gamma\bigl(\frac{i+2}{3}\bigr)}{\,\Gamma\bigl(3i+1\bigr)\,\Gamma\bigl(\frac{4i+2}{3}\bigr)}, 
\label{eq:det33f} \\
E_{1, 0, 1, 1}(n) &= 2\prod_{i=1}^{n} \frac{2^{i-2}\,\Gamma\bigl(4i-1\bigr)\,\Gamma\bigl(\frac{i}{3}\bigr)}{3\,\Gamma\bigl(3i-1\bigr)\,\Gamma\bigl(\frac{4i}{3}\bigr)}, 
\label{eq:det33g} \\
E_{2, 0, 2, 2}(n) &= 2\prod_{i=1}^{n} \frac{2^{i-3}\,\Gamma\bigl(4i+1\bigr)\,\Gamma\bigl(\frac{i+2}{3}\bigr)}{\Gamma\bigl(3i+1\bigr)\,\Gamma\bigl(\frac{4i+2}{3}\bigr)}. 
\label{eq:det33h}
\end{align}
Moreover, some related determinants can be expressed in terms of these;
the following identities hold (at least) for all $n\geq3$:
\begin{align}
E_{0, 0, 0, 0}(n) &= \frac12 \, E_{0, 1, -1, -3}(n) =
\frac15 \, E_{0, 2, -2, -6}(n), \label{eq:det33i} \\
E_{1, 0, 1, 1}(n) &= -\frac{1}{84} \, E_{1, 3, -2, -8}(n) = 2 \, E_{4, 2, 4, 0}(n-1) =
\frac65 \, E_{4, 3, 3, -3}(n-1), \label{eq:det33j} \\
E_{2, 0, 2, 2}(n) &= 2 \, E_{5, 2, 5, 1}(n-1) = 18 \, E_{8, 4, 8, 0}(n-2) =
\frac{162}{5} \, E_{8, 5, 7, -3}(n-2), \label{eq:det33k} \\
E_{-3, 2, 1, -3}(n) &= E_{0, 3, 5, -1}(n-1), \label{eq:det33l} \\
E_{0, 1, -1, -3}(n) &= 4 \, E_{3, 2, 3, -1}(n-1), \label{eq:det33m} \\
E_{1, 1, 0, -2}(n) &= -2 \, E_{4, 2, 4, 0}(n-1), \label{eq:det33n} \\
E_{1, 2, -1, -5}(n) &= -12 \, E_{4, 3, 3, -3}(n-1) = -180 \, E_{7, 4, 7, -1}(n-2),
\label{eq:det33o} \\
E_{2, 1, 1, -1}(n) &= E_{5, 2, 5, 1}(n-1) ,\label{eq:det33p} \\
E_{2, 2, 0, -4}(n) &= \frac{15}{2} \, E_{5, 3, 4, -2}(n-1) =
-45 \, E_{8, 4, 8, 0}(n-2), \label{eq:det33q} \\
E_{2, 3, -1, -7}(n) &= 36 \, E_{5, 4, 3, -5}(n-1) =
-\frac{13608}{5} \, E_{8, 5, 7, -3}(n-2). \label{eq:det33r} 
\end{align}
\end{theorem}
\begin{proof}
Identities \eqref{eq:det33a}--\eqref{eq:det33k} can be proven by the holonomic
Ansatz, see~\cite{EM} for the details. Some computational data are given in
Table~\ref{tab:comp}. Identities \eqref{eq:det33l}--\eqref{eq:det33r} can
easily be established by exploiting the block structure of the corresponding
matrices: the larger matrices in each formula have a block of zeros, and the
smaller matrices from the same formula in the lower right corner.
\end{proof}

\begin{conjecture} \label{conj:det33x}
Let
\[
  \Xi(x) := \prod_{i=2}^x \frac{3 \, \Gamma(i) \, \Gamma(4i-3) \, \Gamma(4 i-2)}{%
    2 \, \Gamma(3i-2)^2 \, \Gamma(3i-1)}
  \qquad\text{and}\qquad
  \mu_m(x) := \begin{cases}
    2, & \text{if\/ } 3 \mid (x-m), \\
    1, & \text{otherwise}.
  \end{cases}
\]
Then, for all non-negative integers~$x$ and for all $n\geq x$, we have
\begin{align}
  E_{0,x,-x,-3x}(n) &=
  2\mu_1(x) \, \Xi(x) \, (-1)^{\lfloor\frac{x}{3}\rfloor}
  \prod_{i=1}^n \frac{2^{i-1} \, \Gamma(4 i-3) \, \Gamma\bigl(\frac{i+1}{3}\bigr)}{%
    \Gamma(3i-2) \, \Gamma\bigl(\frac{4i-2}{3}\bigr)},
  \label{det3j0x}
  \\
  E_{1,x,1-x,1-3x}(n) &=
  2\mu_2(x) \, \Xi(x) \, (-1)^{\lfloor\frac{x+2}{3}\rfloor}
  \prod_{i=1}^n \frac{2^{i-2} \, \Gamma(4i-1) \, \Gamma\bigl(\frac{i}{3}\bigr)}{%
    3 \, \Gamma(3i-1) \, \Gamma\bigl(\frac{4i}{3}\bigr)},
  \label{det3j1x}
  \\
  E_{2,x,2-x,2-3x}(n) &=
  \frac{\mu_0(x)}{n} \, \Xi(x) \, (-1)^{\lfloor\frac{x+1}{3}\rfloor}
  \prod_{i=2}^n \frac{2^{i-3} \, \Gamma(4i+1) \, \Gamma\bigl(\frac{i-1}{3}\bigr)}{%
    9 \, \Gamma(3i) \, \Gamma\bigl(\frac{4i+2}{3}\bigr)}.
  \label{det3j2x}
\end{align}
\end{conjecture}

\begin{remark}
Identity \eqref{det3j0x} generalizes \eqref{eq:det33i} (their closed forms are
obtained by combining \eqref{eq:det33m} with \eqref{eq:det33f}).
Identity~\eqref{det3j1x} generalizes some determinants given in
\eqref{eq:det33j}, \eqref{eq:det33n}, and \eqref{eq:det33o}.
Identity~\eqref{det3j2x} generalizes some determinants given in
\eqref{eq:det33k}, \eqref{eq:det33p}, \eqref{eq:det33q}, and
\eqref{eq:det33r}.
\end{remark}

\section{Variations on the theme, III}
\label{sec:var3}

In this section, we present several variations of the determinant evaluations
in Section~\ref{sec:conj2} in which the power~$2^i$ gets replaced by~$4^i$.
As in the previous sections, we start by identifying some sporadic cases,
which were found in an automated search inside the parameter
space $-6\leq\alpha,\beta\leq9$ and $-9\leq\gamma,\delta\leq9$,
before we turn to two parametric families. We are able to prove one of them
using the holonomic Ansatz; see Theorem~\ref{thm:detx41}.
The second, Theorem~\ref{thm:MS1},
does not seem suitable for the application
of the holonomic Ansatz. On the other hand, the application of a
--- non-algorithmic --- method is feasible: identification of factors.
Due to its length, we provide the corresponding proof separately in the
next section.
Still, this second result must be considered as incomplete as we are not able
to identify one factor in the determinant evaluation; we are only
able to provide a conjectural recurrence that this factor seems to
satisfy; see Conjecture~\ref{conj:MS1}.

\medskip
Let us introduce the following notation for the determinants in question:
\[
  F_{\alpha,\beta,\gamma,\delta}(n) :=
  \det_{0\leq i,j\le n-1}\left(
  4^{i+\beta}\binom{i+2j+\gamma}{2j+\alpha}
  +\binom{-i+2j+\delta}{2j+\alpha}\right).
\]

\begin{theorem}\label{thm:det24}
The following determinant evaluations hold for all $n\geq1$:
\begin{align}
F_{1, 0, 1, 1}(n) &= 2\prod_{i=1}^{n} \frac{3^{i-1}\,\Gamma\bigl(3i-1\bigr)\,\Gamma\bigl(\frac{i+1}{2}\bigr)}{\Gamma\bigl(2i\bigr)\,\Gamma\bigl(\frac{3i-1}{2}\bigr)}, \label{eq:det24a} 
\\
F_{1, 0, 2, 2}(n) &= 2\prod_{i=1}^{n} \frac{3^{i-1}\,\Gamma\bigl(3i\bigr)\,\Gamma\bigl(\frac{i}{2}\bigr)}{2\,\Gamma\bigl(2i\bigr)\,\Gamma\bigl(\frac{3i}{2}\bigr)}, \label{eq:det24b} 
\\
F_{1, 0, 3, 3}(n) &= 2\prod_{i=1}^{n} \frac{3^i\,\Gamma\bigl(3i-1\bigr)\,\Gamma\bigl(\frac{i+1}{2}\bigr)}{\Gamma\bigl(2i\bigr)\,\Gamma\bigl(\frac{3i-1}{2}\bigr)}. \label{eq:det24c} 
\end{align}
Moreover, some related determinants can be expressed in terms of these;
the following identities hold (at least) for all $n\geq4$:
\begin{align}
  F_{1, 0, 1, 1}(n) &= \frac23 \, F_{1, 1, -1, -3}(n) = \frac{1}{21} \, F_{1, 2, -3, -7}(n),
  \label{eq:det24d} \\
  F_{1, 0, 2, 2}(n) &= -2 \, F_{1, 1, 0, -2}(n) = \frac27 \, F_{1, 2, -2, -6}(n), \label{eq:det24e} \\
  F_{1, 0, 3, 3}(n) &= 2 \, F_{1, 1, 1, -1}(n) = \frac25 \, F_{1, 2, -1, -5}(n) =
  \frac{1}{99} \, F_{1, 3, -3, -9}(n), \label{eq:det24f} \\
  F_{1, 1, -1, -3}(n) &= -6 \, F_{3, 2, 2, -2}(n-1) = 24 \, F_{5, 3, 5, -1}(n-2), \label{eq:det24g} \\
  F_{1, 1, 0, -2}(n) &= -2 \, F_{3, 2, 3, -1}(n-1). \label{eq:det24h}
\end{align}
\end{theorem}
\begin{proof}
Identities \eqref{eq:det24a}--\eqref{eq:det24f} can be proven, quite effortlessly,
by the holonomic Ansatz, see~\cite{EM}. For the determinants on the right-hand
sides of \eqref{eq:det24d}--\eqref{eq:det24f}
we have established closed forms, from which the displayed relations follow.
Identities \eqref{eq:det24g}--\eqref{eq:det24h} can easily be established by
exploiting the block structure of the matrices $F_{1, 1, -1, -3}(n)$
respectively 
$F_{3, 2, 2, -2}(n)$ and $F_{1, 1, 0, -2}(n)$, which have a block of zeros
(of size $2\times(n-2)$ respectively $1\times(n-1)$) in their upper
right corner. 
\end{proof}

The parameters of the determinants
in \eqref{eq:det24a}--\eqref{eq:det24c} follow
an obvious pattern (in contrast to their right-hand sides). Indeed, the
determinants $F_{1,0,4,4}(n),\dots,F_{1,0,9,9}(n)$ were also found to
factor nicely, and in fact one can come up with a general closed form.
Note that the determinant below corresponds to $F_{1,0,x+1,x+1}(n)$.

\begin{theorem}\label{thm:detx41}
  Let $x$ be an indeterminate. Then, for all integers $n\geq1$, we have:
  \begin{multline}\label{detx41}
    \det_{0\leq i,j\le n-1}\biggl(
    4^i \binom{x+i+2j+1}{2j+1} + \binom{x-i+2j+1}{2j+1}\biggr) =
    2\prod_{i=1}^n \frac{2^{2i-1} \, 3^{i-1} \, \Gamma(i) \, \Gamma\bigl(\frac{3i+x}{2}\bigr)}{%
       \Gamma(2i) \, \Gamma\bigl(\frac{i+x}{2}\bigr)}
\\
=
2^{\binom {n+1}2+1}3^{\binom n2}
\prod _{i=1} ^{n}\frac {i!} {(2i)!}
\prod _{i=0} ^{n-1}(x+3i+1)_{n-i}.
\end{multline}
\end{theorem}

\begin{proof}
  The proof is analogous to the proof of Theorem~\ref{thm:det24}, with the
  only difference that the computations are heavier, due to the additional
  parameter~$x$. Since among all determinants in this paper, the ones stated
  in Theorem~\ref{thm:det24} require the least computational effort, their
  parameterized version~\eqref{detx41} is still doable, while all other
  parameterized determinants resisted a proof via the holonomic Ansatz, due to
  their computational complexity (compare the data given in
  Table~\ref{tab:comp}).
\end{proof}

\begin{theorem} \label{thm:MS1}
For all positive integers $n$, we have
\begin{multline} \label{eq:MS1}
\det_{0\le i,j\le n-1}\left(4^i\binom {x+i+2j+3}{2j+3}+\binom {x-i+2j+3}{2j+3}\right)
\\
=\left(2\cdot 6^{\binom n2}
\prod _{i=0} ^{n-1}\frac {i!} {(2i+3)!}
\right)
\left((x+2)(x+3)
\prod _{i=0} ^{n-1}(x+3i+1)_{n-i}\right)
\times\text{\rm Pol}_n(x),
\end{multline}
where $\text{\rm Pol}_n(x)$ is a monic polynomial in $x$ of
degree~$2n-2$.
\end{theorem}

The proof of this theorem is given in the next section.

As already mentioned, we do not know an explicit formula for the polynomials
$\text{Pol}_n(x)$ but, experimentally, we found a recurrence that they
seem to satisfy.

\begin{conjecture} \label{conj:MS1}
The polynomial $\text{\rm Pol}_n(x)$ in Theorem~\ref{thm:MS1}
is given by the recurrence
\begin{multline*}
   3 \, \text{\rm Pol}_{n+3}(x)
  -2 \bigl(18 n^2+9 n x+72 n-3 x^2-3 x+49\bigr)\, \text{\rm Pol}_{n+2}(x) \\
   + \bigl(135 n^4+108 n^3 x+810 n^3-54 n^2 x^2+108 n^2 x+1395 n^2-52 n x^3-510 n x^2 \\
  -1100 n x +120 n-9 x^4-152 x^3-855 x^2-1780x-1020\bigr)\, \text{\rm Pol}_{n+1}(x) \\
   -6 (n+1) (n-x-2) (n+x+2) (3 n+x+3) (3 n+x+5) (3 n+x+7)\, \text{\rm Pol}_n(x) = 0
\end{multline*}
and initial values
\begin{align*}
  \text{\rm Pol}_1(x) &= 1, \\
  \text{\rm Pol}_2(x) &= \frac13 \bigl(3 x^2+31 x+60\bigr), \\
  \text{\rm Pol}_3(x) &= \frac19 \bigl(9 x^4+234 x^3+2061 x^2+6956x+7680\bigr).
\end{align*}
\end{conjecture}

\section{Proof of Theorem \ref{thm:MS1}}
\label{sec:proof}

We now provide our proof of Theorem~\ref{thm:MS1}. Essential parts of
it are based on several auxiliary results that, for the sake of better
readability, are stated and proved separately in
Lemmas~\ref{lem:CK3}--\ref{lem:CK6} further below. 

\begin{proof}[Proof of Theorem \ref{thm:MS1}]
Let us denote the matrix on the left-hand side of~\eqref{eq:MS1}
of which we take the determinant by~$D_n(x)$.

We proceed in several steps. 
First we show that the linear factors that appear on\break
the right-hand
side of~\eqref{eq:MS1} are indeed polynomial factors of~$\det D_n(x)$;
see\break Steps~1--6 below.
In Step~7, we show that the degree
of~$\det D_n(x)$ as a polynomial in~$x$ is bounded above by $\binom
{n+1}2+2n$. Since the prefactor of $\text{Pol}_n(x)$ on the right-hand
side of~\eqref{eq:MS1} has degree $2+\sum_{i=0}^{n-1}(n-i)=2+\binom
{n+1}2$, this implies that the degree of $\text{Pol}_n(x)$ is at
most~$2n-2$. We complete the proof by Step~8 in which we compute the
coefficient of $x^{\binom {n+1}2+2n}$ in the determinant~$\det D_n(x)$, and
thus the leading coefficient of both~$\det D_n(x)$ and~$\text{Pol}_n(x)$.

Below, we use the truth function $\chi$ which is defined by
$\chi(\mathcal A)=1$ if $\mathcal A$ is true and $\chi(\mathcal A)=0$
otherwise. 

We start with some simple divisibility properties of $\det D_n(x)$.

\medskip
{\sc Step 1.} {\it $(x+1)$ is a factor of\/ $\det D_n(x)$.}
This is seen by noting that $(x+1)$ is a factor of each entry
in row~0 of~$D_n(x)$.

\medskip
{\sc Step 2.} {\it $(x+2)^{1+\chi(n\ge2)}$ is a factor of\/ $\det D_n(x)$.}
On the one hand, $(x+2)$ is also a factor of each entry
in row~0 of~$D_n(x)$.
Moreover, $(x+2)$ is a factor of each entry in row~1.

\medskip
{\sc Step 3.} {\it $(x+3)^{1+\chi(n\ge3)}$ is a factor of\/ $\det D_n(x)$.}
Similarly, also $(x+3)$ is a factor of each entry
in row~0 of~$D_n(x)$.
On the other hand, $(x+3)$ is also a factor of each entry in row~1 and
row~2, except for the entries in column~0, which are
$4\binom {x+4}3+\binom {x+2}3$ and 
$16\binom {x+5}3+\binom {x+1}3$, respectively. One can check
that 4~times the first expression minus the second yields a polynomial
that is divisible by~$(x+3)$. By an elementary row operation,
this implies that, as soon as $n\ge3$,
another term $(x+3)$ divides the determinant~$\det D_n(x)$.

\medskip
Before we continue with the ``general" case, we need a few
preparations. By inspection of the right-hand side of~\eqref{eq:MS1},
we see that it remains to show that for $4\le \beta\le 3n-2$ the term
\begin{align} \label{eq:beta}
\notag
(x+\beta)^{\#\{i\ge0:3i+1\le \beta\le n+2i\}}
&=(x+\beta)^{\fl{(\beta-1)/3}-\max\{0,\cl{(\beta-n)/2}\}+1}\\
&=(x+\beta)^{\min\{\fl{(\beta+2)/3},\fl{(\beta+2)/3}-\cl{(\beta-n)/2}\}}
\end{align}
divides $\det D_n(x)$.
We are going to do this by applying the idea of ``identification of
factors" as described in Section~2.4 of~\cite{KratBN}. To be precise,
in order to prove that $(x+\beta)^E$ is a factor of~$\det D_n(x)$, we find
$E$~linear combinations of rows of~$D_n(x)$ that vanish and that are
linearly independent. (In other words, we find $E$~linearly
independent vectors in the left kernel of the matrix~$D_n(x)$.
That the latter is indeed sufficient to infer the claimed divisibility
is argued in \cite[Sec.~2]{KratBI}.)

Our description of these linear combinations of rows of~$D_n(x)$ is in
terms of generating functions, in complete analogy to the calculus
that we applied in Section~\ref{sec:conj2}. Namely, by~\eqref{eq:GF1}
and~\eqref{eq:GF2} the generating function for the entries in row~$i$
of our matrix~$D_n(x)$ is
\begin{align} \notag
\sum_{j\ge0}&\left(4^i\binom {x+i+2j+3}{2j+3}+\binom {x-i+2j+3}{2j+3}\right)v^{j}\\
\notag
&\kern0cm
=4^i\cdot\frac {1} {v}
\left(\frac {1} {2\vv}\left((1-\vv)^{-x-i-1}-(1+\vv)^{-x-i-1}\right)-(x+i)\right)\\
\notag
&\kern1cm
+\frac {1} {v}
\left(\frac {1}
     {2\vv}\left((1-\vv)^{-x+i-1}-(1+\vv)^{-x+i-1}\right)-(x-i)\right)\\
\notag
&\kern0cm
=\frac {1} {2v^{3/2}}
\Big(
4^i\left((1-\vv)^{-x-i-1}-(1+\vv)^{-x-i-1}\right)\\
&\kern2cm
+\left((1-\vv)^{-x+i-1}-(1+\vv)^{-x+i-1}\right)\Big)
-\frac {1} {v}\left(4^i(x+i)+(x-i)\right).
\label{eq:GFRi} 
\end{align}
In view of this expression, Lemma~\ref{lem:CK5} says that, for
non-negative real numbers~$s$ and~$t$ with $t\le s$ such
that all of $s+2t$, $2s+t$, $3s$, and~$3t$ are integers, we have
\begin{equation} \label{eq:lin2} 
\sum_{i=0}^{s+2t}(-1)^i2^{s+2t-i}\alpha(i)
\left(\sum_{j=0}^{3t}\binom {3t}j
\binom {2s+t-j}{s+2t-2j-i}\right)
\cdot(\text{row $i$ of $D_n(-3s-1)$})=0,
\end{equation}
where $\alpha(i)=1$ if $i>0$ and $\alpha(0)=\frac {1} {2}$.
Indeed, Lemma~\ref{lem:CK5} implies that, when we apply generating
function calculus to prove~\eqref{eq:lin2} using~\eqref{eq:GFRi},
all powers of $(1-\sqrt v)$ and $(1+\sqrt v)$ cancel out. On the other
hand, it seems that we would have to check that also the terms that
result from the expressions $\frac {1}
{v}\left(4^i(x+i)+(x-i)\right)$ on the right-hand side
of~\eqref{eq:GFRi} cancel out. That could certainly be done by
computing the corresponding binomial sums. However, it comes for free:
we use the generating function in~\eqref{eq:GFRi} with $x=-3s-1$
and~$i$ in the range $0\le i\le 2s+t$; this implies that
$$x+i=-3s-1+i\le -3s-1+2s+t=-s+t-1\le-1$$
by one of the assumptions of
Lemma~\ref{lem:CK5}. Therefore, with these choices of~$x$
and~$i$, both binomial coefficients in the sum on left-hand side
of~\eqref{eq:GFRi} vanish for large enough~$j$.
In other words: the generating function in~\eqref{eq:GFRi} is always a
polynomial in~$v$. Hence, terms involving negative powers of~$v$ must
automatically cancel out. 

We now discuss the divisibility of $\det D_n(x)$ by the power in~\eqref{eq:beta}
for the congruence classes of~$\beta$ modulo~3 separately.

\medskip
{\sc Step 4.} {\it
  $(x+3s+1)^{\min\{s+1,s+1-\cl{(3s+1-n)/2}\}}$ is a factor
of\/ $\det D_n(x)$ for $1\le s\le n-1$.} The linear combinations of
rows of $D_n(-3s-1)$ given in~\eqref{eq:lin2} vanish for $0\le t\le s$. They are
linearly independent since the highest row number involved is $s+2t$,
which is different for different~$t$. Another restriction that must be taken
into account is that we may only use actually existing rows
of~$D_n(x)$, meaning that we must have $s+2t\le n-1$. In summary, the
number of vanishing linear combinations~\eqref{eq:lin2} of rows, or,
equivalently, the number of integers~$t$ with $0\le t\le s$ and
$s+2t\le n-1$, equals $\min\{s+1,\fl{(n-1-s)/2}+1\},$
which agrees with the claimed exponent.

\medskip
{\sc Step 5.} {\it
  $(x+3s+2)^{\min\{s+1,s+1-\cl{(3s+2-n)/2}}$ is a factor
of\/ $\det D_n(x)$ for $1\le s\le n-2$.}
We use~\eqref{eq:lin2} with $s$ replaced by $s+\frac {1} {3}$ and $t$
replaced by $t+\frac {1} {3}$. The conclusion is that
the linear combinations of
rows of $D_n(-3s-2)$ given by~\eqref{eq:lin2} vanish for $0\le t\le s$. They are
linearly independent since the highest row number involved is $s+2t+1$,
which is different for different~$t$. Another restriction that must be taken
into account is that we may only use actually existing rows
of~$D_n(x)$, meaning that we must have $s+2t+1\le n-1$. In summary, the
number of vanishing linear combinations~\eqref{eq:lin2} of rows, or,
equivalently, the number of integers~$t$ with $0\le t\le s$ and
$s+2t+1\le n-1$, equals $\min\{s+1,\fl{(n-2-s)/2}+1\}$,
which agrees with the claimed exponent.

\medskip
{\sc Step 6.} {\it
  $(x+3s+3)^{\min\{s+1,s+1-\cl{(3s+3-n)/2}\}}$ is a factor
of\/ $\det D_n(x)$ for $1\le s\le n-2$.}
We use~\eqref{eq:lin2} with $s$ replaced by $s+\frac {2} {3}$ and $t$
replaced by $t+\frac {2} {3}$. The conclusion is that
the linear combinations of
rows of $D_n(-3s-3)$ given in~\eqref{eq:lin2} vanish for $0\le t\le s$. They are
linearly independent since the highest row number involved is $s+2t+2$,
which is different for different~$t$. Another restriction that must be taken
into account is that we may only use actually existing rows
of~$D_n(x)$, meaning that we must have $s+2t+2\le n-1$. In summary, the
number of vanishing linear combinations~\eqref{eq:lin2} of rows, or,
equivalently, the number of integers~$t$ with $0\le t\le s$ and
$s+2t+2\le n-1$, equals $\min\{s+1,\fl{(n-3-s)/2}+1\}$,
which agrees with the claimed exponent.

\medskip
{\sc Step 7.}
{\it $\det D_n(x)$ is a polynomial in $x$ of degree at most $\binom {n+1}2+2n$.}
To see this, we replace column~$j$ of the matrix by
$$
\sum_{k=0}^j(-1)^{j-k}\binom jk(2k+3)!\,
(-x)_{2j-2k}\cdot(\text{column $k$ of the matrix}).
$$
Clearly, this can be achieved by elementary column operations.
Thereby, the $j$-th column is multiplied by $(2j+3)!$, and therefore
the new determinant equals $\det D_n(x)$ multiplied by 
\begin{equation} \label{eq:mult} 
\prod _{i=0} ^{n-1}(2i+3)!.
\end{equation}
Let $M(n)$ denote the new matrix.
The $(i,j)$-entry of $M(n)$ then is
$$
\sum_{k=0}^j(-1)^{j-k}\binom jk(2k+3)!\,(-x)_{2j-2k}
\left(4^i\binom {x+i+2k+3}{2k+3}+\binom {x-i+2k+3}{2k+3}\right).
$$
Using the standard hypergeometric notation
$$
{}_r F_s\!\left[\begin{matrix} a_1,\dots,a_r\\ b_1,\dots,b_s\end{matrix};  
z\right]=\sum _{l =0} ^{\infty}\frac {\po{a_1}{l }\cdots\po{a_r}{l }} 
{l !\,\po{b_1}{l }\cdots\po{b_s}{l }} z^l \ , 
$$
we have
\begin{multline} \label{eq:AA}
\sum_{k=0}^j(-1)^{j-k}\binom jk(2k+3)!\,(-x)_{2j-2k}\binom {x+i+2k+3}{2k+3}
\\=
(-1)^j(-x)_{2j}\,(x+i+1)_3\cdot
{} _{3} F _{2} \!\left [ \begin{matrix} 
\frac {x} {2}+\frac {i} {2}+\frac {5} {2},\frac {x} {2}+\frac {i} {2}+2,-j\\ 
\frac {x} {2}-j+1,\frac {x} {2}-j+\frac {1} {2}
\end{matrix} ; {\displaystyle 1}\right ].
\end{multline}
To this $_3F_2$-series we apply the transformation formula
(see \cite[Eq.~(3.1.1)]{GaRaAF})
$$
{} _{3} F _{2} \!\left [ \begin{matrix} { a, b, -n}\\ { d, e}\end{matrix} ;
   {\displaystyle 1}\right ]  =
{\frac { ({ \textstyle e-b}) _{n} }
      {({ \textstyle e}) _{n} }}  
{} _{3} F _{2} \!\left [ \begin{matrix} { -n, b, d-a}\\ { d, 1 + b - e -
       n}\end{matrix} ; {\displaystyle 1}\right ] ,
$$
where $n$ is a non-negative integer. Thus, we obtain 
\begin{align*}
\sum_{k=0}^j(-1)^{j-k}&\binom jk(2k+3)!\,(-x)_{2j-2k}\binom {x+i+2k+3}{2k+3}
\\
&=
(-1)^j
\frac {(x+i+1)_3\,(-x)_{2j}\,
(-j-\frac {i} {2}-\frac {3} {2})_j} {(\frac {x} {2}-j+\frac {1} {2})_j}
{} _{3} F _{2} \!\left [ \begin{matrix} 
-j,\frac {x} {2}+\frac {i} {2}+2,-j-\frac {i} {2}-\frac {3} {2}\\
\frac {x} {2}-j+1,\frac {i} {2}+\frac {5} {2}
\end{matrix} ; {\displaystyle 1}\right ]
\\
&=
2^{2j}\,(x+i+1)_3\,
(-j-\tfrac {i} {2}-\tfrac {3} {2})_j\\
&\kern1.5cm
\times
\sum_{k=0}^j(-1)^{j-k}\binom jk
\frac {(\frac {x} {2}-j+k+1)_{j-k}\,(\frac {x} {2}+\frac {i} {2}+2)_k\,(-j-\frac {i} {2}-\frac {3} {2})_k} 
{(\frac {i} {2}+\frac {5} {2})_k}.
\end{align*}
We see that this is a polynomial in $x$ of degree $j+3$, with leading
coefficient
\begin{multline*}
2^{2j}\,
(-j-\tfrac {i} {2}-\tfrac {3} {2})_j
\sum_{k=0}^j(-1)^{j-k}\binom jk
\frac {2^{-j}\,(-j-\frac {i} {2}-\frac {3} {2})_k} 
{(\frac {i} {2}+\frac {5} {2})_k}
\\=
(-1)^j2^{j}\,
(-j-\tfrac {i} {2}-\tfrac {3} {2})_j\,
{} _{2} F _{1} \!\left [ \begin{matrix} 
-j,-j-\frac {i} {2}-\frac {3} {2}\\\frac {i} {2}+\frac {5} {2}
\end{matrix} ; {\displaystyle 1}\right ].
\end{multline*}
The $_2F_1$-series can be evaluated by means of the Chu--Vandermonde
summation (see \cite[Eq.~(1.7.7); Appendix (III.4)]{SlatAC})
$$
{} _{2} F _{1} \!\left [ \begin{matrix} { a, -n}\\ { c}\end{matrix} ; {\displaystyle
   1}\right ]  = {\frac {({ \textstyle c-a}) _{n} }
    {({ \textstyle c}) _{n} }},
$$
where $n$ is a non-negative integer. After simplification, we see
that \eqref{eq:AA} is a polynomial in~$x$ of degree~$j+3$ with leading coefficient
$$
2^j(i+j+4)_j.
$$
In its turn, this implies that the $(i,j)$-entry of $M(n)$ is a polynomial
in~$x$ of degree~$j+3$ with leading coefficient
\begin{equation} \label{eq:leadij} 
2^j\big(4^i(i+j+4)_j+(-i+j+4)_j\big).
\end{equation}
Consequently, the determinant $\det M(n)$ is a polynomial in~$x$ of
degree at most
$$\sum_{j=0}^{n-1}(j+3)=\binom {n+1}2+2n.$$
Since $\det D_n(x)$ is a scalar multiple of $\det M(n)$, the same
degree bound holds for $\det D_n(x)$.

\medskip
{\sc Step 8.} {\it Computation of the leading coefficient of $\text{Pol}_n(x)$.}
In the previous step we found that the degree of $\det D_n(x)$
as a polynomial in~$x$ is at most $\binom
{n+1}2+2n$.
We are now going to show that this is the exact degree, by computing
the coefficient of $x^{\binom {n+1}2+2n}$ in~$\det D_n(x)$, which then is at the same
time the leading coefficient of~$\text{Pol}_n(x)$.

In order to compute this coefficient of $x^{\binom {n+1}2+2n}$, we
should recall from the previous step, that we transformed our original
determinant~$\det D_n(x)$ into (cf.\ the sentence
containing~\eqref{eq:mult})
\begin{equation} \label{eq:detM(n)}
\left(\prod _{i=0} ^{n-1}\frac {1} {(2i+3)!}\right)\det M(n),
\end{equation}
where the $(i,j)$-entry of $M(n)$ is a polynomial in~$x$ of
degree~$j+3$ with leading coefficient given by~\eqref{eq:leadij}.
Hence, the leading coefficient of $\det M(n)$ (the coefficient of
$x^{\binom {n+1}2+2n}$) equals
$$
2^{\binom n2}
\det_{0\le i,j\le n-1}\left(4^i(i+j+2)_j+(-i+j+2)_j\right).
$$
By applying column operations, this expression can be reduced to
$$
2^{\binom n2}
\det_{0\le i,j\le n-1}\left(4^ii^j+(-i)^j\right).
$$
By the same argument, this expression equals
$$
2^{\binom n2}
\det_{0\le i,j\le n-1}\left(4^ip_j(i)+p_j(-i)\right),
$$
where $p_j(t)$ is any monic polynomial in $t$ of degree~$j$.
We choose $p_j(t)=(t)_j$, so that we need to evaluate
\begin{multline} \label{eq:lead2}
2^{\binom n2}\det_{0\le i,j\le n-1}\left(4^i(i)_j+(-i)_j\right)\\
=
2^{\binom n2}\left(
\prod _{j=0} ^{n-1}j!\right)
\det_{0\le i,j\le n-1}\left(4^i\binom {i+j-1}j+\binom {-i+j-1}j\right).
\end{multline}
If we now combine \eqref{eq:detM(n)} and \eqref{eq:lead2}, and subsequently
evaluate the last determinant by means of Theorem~\ref{thm:einfach}
with $a=4$ and $x=0$, then we obtain the first expression in
parentheses on the right-hand side of~\eqref{eq:MS1}
for the leading coefficient of our determinant~$\det D_n(x)$.

\medskip
This completes the proof of the theorem modulo
Lemmas~\ref{lem:CK3}--\ref{lem:CK6} which are stated and proved separately below.
\end{proof}

\begin{lemma} \label{lem:CK3}
For all non-negative integers $s$, we have
$$
\sum_{i=0}^{s}(-1)^i2^{i}\binom s{i}
\left((1-\sqrt v)^{2s-i}-(1+\sqrt v)^{2s-i}\right)=0.
$$
\end{lemma}

\begin{proof}
By the binomial theorem, we get
\begin{align*}
\sum_{i=0}^{s}(-1)^i2^{i}\binom s{i}&
\left((1-\sqrt v)^{2s-i}-(1+\sqrt v)^{2s-i}\right)\\
&=(1-\sqrt v)^{2s}\left(1-\frac {2} {1-\sqrt v}\right)^s
-(1+\sqrt v)^{2s}\left(1-\frac {2} {1+\sqrt v}\right)^s\\
&=(1-\sqrt v)^{2s}\left(\frac {-1-\sqrt v} {1-\sqrt v}\right)^s
-(1+\sqrt v)^{2s}\left(\frac {-1+\sqrt v} {1+\sqrt v}\right)^s\\
&=0,
\end{align*}
as desired.
\end{proof}

\begin{lemma} \label{lem:CK5}
Let
\begin{equation} \label{eq:Syi} 
S(y,i):=4^i\left((1-\sqrt v)^{y-i}-(1+\sqrt v)^{y-i}\right)
+\left((1-\sqrt v)^{y+i}-(1+\sqrt v)^{y+i}\right).
\end{equation}
Then, for all non-negative real numbers~$s$ and~$t$ with $t\le s$ such
that all of $s+2t$, $2s+t$, $3s$, and~$3t$ are integers, we have
\begin{equation} \label{eq:lin1} 
\sum_{i=0}^{s+2t}(-1)^i2^{s+2t-i}\alpha(i)
\left(\sum_{j=0}^{3t}\binom {3t}j
\binom {2s+t-j}{s+2t-2j-i}\right)
S(3s,i)=0,
\end{equation}
where $\alpha(i)=1$ if $i>0$ and $\alpha(0)=\frac {1} {2}$.
\end{lemma}

\begin{remarks}
(1)
The integrality conditions on the parameters~$s$ and~$t$ may seem a
bit contrived. Indeed, these conditions are equivalent to saying that
the pair $(s,t)$ is of the form $(s_1+\frac
{u} {3}, t_1+\frac {u} {3})$, where all of $s_1,t_1,u$ are non-negative integers.
It is exactly in this form in which the lemma is used in Steps~4--6 of the proof
of Theorem~\ref{thm:MS1}. On the other hand, for the proof of the
lemma it is more convenient to have these conditions in this
``contrived" form.

\medskip
(2) The condition $t\le s$ is needed to make sure that the range of
the summation index~$i$ on the left-hand side of~\eqref{eq:lin1}
does not extend beyond~$3s$. The latter would yield negative
powers of $(1-\sqrt v)$ and $(1+\sqrt v)$ in the definition of
$S(3s,i)$, for which the below application of Lemma~\ref{lem:CK3} is
not possible.
\end{remarks}

\begin{proof}[Proof of Lemma \ref{lem:CK5}]
By the definition of $S(y,i)$ in~\eqref{eq:Syi}, the left-hand side
in~\eqref{eq:lin1} equals
\begin{align*}
&\sum_{i=0}^{s+2t}(-1)^i2^{s+2t-i}\alpha(i)
\left(\sum_{j=0}^{3t}\binom {3t}j
\binom {2s+t-j}{s+2t-2j-i}\right)\\
&\kern2cm
\cdot
\left(4^i\left((1-\sqrt v)^{3s-i}-(1+\sqrt v)^{3s-i}\right)
+\left((1-\sqrt v)^{3s+i}-(1+\sqrt v)^{3s+i}\right)\right)\\
&=
\sum_{i=0}^{s+2t}(-1)^i2^{s+2t+i}\alpha(i)
\left(\sum_{j=0}^{3t}\binom {3t}j
\binom {2s+t-j}{s+2t-2j-i}\right)
\left((1-\sqrt v)^{3s-i}-(1+\sqrt v)^{3s-i}\right)\\
&\kern1cm
+\sum_{i=0}^{s+2t}(-1)^i2^{s+2t-i}\alpha(i)
\left(\sum_{j=0}^{3t}\binom {3t}j
\binom {2s+t-j}{s+2t-2j-i}\right)\\
&\kern6cm
\cdot\left((1-\sqrt v)^{3s+i}-(1+\sqrt v)^{3s+i}\right)\\
&=
\sum_{i=0}^{s+2t}(-1)^i2^{s+2t+i}\alpha(i)
\left(\sum_{j=0}^{3t}\binom {3t}j
\binom {2s+t-j}{s+2t-2j-i}\right)
\left((1-\sqrt v)^{3s-i}-(1+\sqrt v)^{3s-i}\right)\\
&\kern1cm
+\sum_{i=-s-2t}^{0}(-1)^i2^{s+2t+i}\alpha(i)
\left(\sum_{j=0}^{3t}\binom {3t}j
\binom {2s+t-j}{s+2t-2j+i}\right)\\
&\kern6cm
\cdot\left((1-\sqrt v)^{3s-i}-(1+\sqrt v)^{3s-i}\right).
\end{align*}
Now we use Lemma~\ref{lem:CK6} to replace the term $-i$ in the first sum
over~$j$ by $+i$. Having done this, the two sums over~$i$ can now be
``concatenated" into one sum,
\begin{align*}
&\sum_{i=-s-2t}^{s+2t}(-1)^i2^{s+2t+i}
\left(\sum_{j=0}^{3t}\binom {3t}j
\binom {2s+t-j}{s+2t-2j+i}\right)
\left((1-\sqrt v)^{3s-i}-(1+\sqrt v)^{3s-i}\right)\\
&\kern.2cm
=
\sum_{j=0}^{3t}\binom {3t}j
\sum_{i=-s-2t}^{s+2t}(-1)^i2^{s+2t+i}
\binom {2s+t-j}{s+2t-2j+i}
\left((1-\sqrt v)^{3s-i}-(1+\sqrt v)^{3s-i}\right)\\
&\kern.2cm
=
\sum_{j=0}^{3t}\binom {3t}j
\sum_{i=0}^{2s+t-j}(-1)^i2^{i+2j}
\binom {2s+t-j}{i}
\left((1-\sqrt v)^{4s+2t-2j-i}-(1+\sqrt v)^{4s+2t-2j-i}\right),
\end{align*}
where we performed the shift of index $i\mapsto i-s-2t+2j$ to obtain
the last line.
For fixed~$j$, the inner sum over~$i$ vanishes due to
Lemma~\ref{lem:CK3} with $s$ replaced by $2s+t-j$. This proves the
assertion of the lemma.
\end{proof}

\begin{lemma} \label{lem:CK6}
For all non-negative integers $s$, $t$, and $i$, the sum
\begin{equation} \label{eq:sum} 
\sum_{j=0}^{3t}\binom {3t}j
\binom {2s+t-j}{s+2t-2j-i}
\end{equation}
is invariant under the replacement $i\mapsto-i$.
\end{lemma}

\begin{proof}
We write the sum in \eqref{eq:sum} in terms of a complex contour
integral. We have
$$
\sum_{j=0}^{3t}\binom {3t}j
\binom {2s+t-j}{s+2t-2j-i}
=\sum_{j=0}^{3t}\frac {1} {(2\pi \ii)^2}\int_{C_x}\int_{C_y}
\frac {(1+y)^{3t}} {y^{3t-j+1}}
\frac {(1+x)^{2s+t-j}} {x^{s+2t-2j-i+1}}\,dy\,dx,
$$
where $C_x$ and $C_y$ are contours encircling the origin once in positive
direction. We assume in both cases that the contours are strictly
contained in the unit disk which has the origin as centre.
In the formula above, $\ii$ stands for $\sqrt{-1}$.

We may extend the sum to all non-negative~$j$ because this only
adds vanishing terms. Moreover, since we assumed that along the
contours the moduli of $x$ and~$y$ are always strictly less than~1, we
may interchange integrals and sum and then evaluate the arising
geometric series. The conclusion is that the sum in~\eqref{eq:sum} is
equal to
$$
\frac {1} {(2\pi \ii)^2}\int_{C_x}\int_{C_y}
\frac {(1+y)^{3t}} {y^{3t+1}}
\frac {(1+x)^{2s+t}} {x^{s+2t-i+1}}
\frac {1} {1-\frac {x^2y} {1+x}}\,dy\,dx.
$$
Now we may blow up the contour $C_y$. We will pick up a residue
at the singularity $y=\frac {1+x} {x^2}$. On the other hand, since
(for fixed~$x$) the integrand is of the order $O(y^{-2})$ as
$|y|\to\infty$, the limit of the integral as the contour tends
to infinity vanishes. In summary, this leads to the expression
$$
\frac {1} {2\pi \ii}\int_{C_x}
\frac {\left(1+\frac {1+x} {x^2}\right)^{3t}}
      {\left(\frac {1+x} {x^2}\right)^{3t+1}}
\frac {(1+x)^{2s+t+1}} {x^{s+2t-i+3}}
\,dx
=
\frac {1} {2\pi \ii}\int_{C_x}
\frac {\left(1+x+x^2\right)^{3t}(1+x)^{2s-2t}} {x^{s+2t-i+1}}
\,dx
$$
for the sum in~\eqref{eq:sum}.

Our task is to show that the last expression is invariant under the
replacement\break $i\mapsto-i$. Indeed, the substitution $x\mapsto
\frac {1} {x}$ turns this expression into itself with $+i$ in place
of~$-i$. This completes the proof of the lemma.
\end{proof}

\begin{remark}
It would be possible, using Lemmas~\ref{lem:CK3}--\ref{lem:CK6}, to
provide an alternative proof of Theorem~\ref{thm:detx41}.
We are also convinced that proofs in a similar style of
Theorems~\ref{thm:D3} and~\ref{thm:D4} are possible.
In its turn, via the determinantal relations established in
Lemmas~\ref{lem:CK1} and~\ref{lem:CK2}, this would
yield new proofs of the enumerative results in~\cite{CoHuKr23}.
\end{remark}

\section{Variations on the theme, IV}
\label{sec:var5}

We conclude with further variations of Theorem~\ref{thm:DiFran}.
Here, the power $2^i$ remains unchanged, but
the terms $2j$ in the binomials are replaced
by~$4j$. Conjecture~\ref{conj:4j} contains the determinant
evaluation of this type that we found experimentally which does not
contain any shifts.
We applied the holonomic Ansatz, and we are confident that it
would go through once our computers are ``strong" enough to carry out
the necessary calculations. At this point in time, however, we must
leave the determinant evaluation as a conjecture.
Moreover, we performed again an automated search for determinant
evaluations where shifts are allowed. This led to the --- again
conjectural --- discovery of many more determinant evaluations; see
Proposition~\ref{prop:4j} and
Conjecture~\ref{conj:det42}. The same remark applies here:
we are confident that all of these could be proved by the holonomic Ansatz
once our computers dispose of sufficient computational power.

\begin{conjecture} \label{conj:4j}
For all positive integers $n$, we have
\begin{multline}
\det_{0\le i,j\le n-1}\left(2^i\binom {i+4j+3}{4j+3}
+\binom {-i+4j+3}{4j+3}\right)
\\
=
\frac {2^{n^2 - n + 1} 3^{2 n} 5^{-\frac58 n^2 + \frac54 n}
(2 n)!\, (\frac 23)_n 
\displaystyle\prod _{i=1} ^{n}\frac {(6 i - 4)!}{(5 i)!} }
{\displaystyle 
\prod _{i=1} ^{\fl{(n+3)/4}}\textstyle(\frac 15)_{ n + 3 - 4 i}
\displaystyle 
\prod _{i=1} ^{\fl{(n+2)/4}}\textstyle(\frac 25)_{ n + 2 - 4 i}
\displaystyle 
\prod _{i=1} ^{\fl{(n+1)/4}}\textstyle(\frac 35)_{ n + 1 - 4 i}
\displaystyle 
\prod _{i=1} ^{\fl{n/4}}\textstyle(\frac 45)_{ n - 4 i}}\\
\times
\begin{cases} 
\displaystyle 5^{3/8},
&\text{for }n=4m-3,\\
\displaystyle 1,
&\text{for }n=4m-2,\\
\displaystyle 5^{-1/8},
&\text{for }n=4m-1,\\
\displaystyle 1,
&\text{for }n=4m.
\end{cases} \label{eq:conj4j}
\end{multline}
\end{conjecture}

As in previous sections, we performed a systematic search for determinants
of the same form. Let us denote
\[
  G_{\alpha,\beta,\gamma,\delta}(n) :=
  \det_{0\leq i,j\le n-1}\left(
  2^{i+\beta}\binom{i+4j+\gamma}{4j+\alpha}
  +\binom{-i+4j+\delta}{4j+\alpha}\right).
\]
In the parameter space $-6\leq\alpha,\beta\leq9$ and
$-9\leq\gamma,\delta\leq9$, we have identified $18$ cases of determinants that
factor completely. Unfortunately, we were not able to prove their conjectured
evaluations, but at least we can state some simple relationships.

\begin{prop} \label{prop:4j}
For all integers $n\geq2$ we have the following relations:
\begin{align}
G_{0, 1, -2, -4}(n) &= 3 \, G_{4, 2, 3, -1}(n-1) = 3 \, G_{8, 3, 8, 2}(n-2), \\
G_{1, 1, 0, -2}(n) &= -2 \, G_{5, 2, 5, 1}(n-1), \\
G_{3, 3, 2, -4}(n) &= -20 \, G_{7, 4, 7, -1}(n-1).
\end{align}
\end{prop}
\begin{proof}
These identities can easily be established by exploiting the block structure
of the matrices $G_{0, 1, -2, -4}(n)$ respectively $G_{4, 2, 3,
  -1}(n)$, $G_{1, 1, 
  0, -2}(n)$, $G_{3, 3, 2, -4}(n)$, which have a block of zeros (of size
$2\times(n-2)$ respectively $1\times(n-1)$) in their upper right corner.
\end{proof}

\begin{conjecture} \label{conj:det42}
The following determinant evaluations hold for all $n\geq1$:
\begin{align}
G_{0, 2, 3, -1}(n) &= \prod_{i=1}^{n} \frac{(2i-1)(4i-3)(4i-1)\,
  \Gamma\bigl(6i\bigr)\,\Gamma\bigl(\frac{i+3}{4}\bigr)}{i(i+1)(i+2)(3i-1)\,
  \Gamma\bigl(5i-1\bigr)\,\Gamma\bigl(\frac{5i+3}{4}\bigr)}, \label{eq:det42a} 
\\
G_{1, 3, 6, 0}(n) &= \prod_{i=1}^{n} \frac{8(2i-1)(2 i+1)^2(4i-1)(4i+1)\,
  \Gamma\bigl(6i+2\bigr)\,\Gamma\bigl(\frac{i+2}{4}\bigr)}{(i+1)(i+2)(i+3)(i+4)\,
  \Gamma\bigl(5i+2\bigr)\,\Gamma\bigl(\frac{5i+6}{4}\bigr)}, \label{eq:det42b} 
\\
G_{1, 1, 0, -2}(n) &= -4\prod_{i=1}^{n} \frac{(3i-2)\,
  \Gamma\bigl(6i-5\bigr)\,\Gamma\bigl(\frac{i}{4}\bigr)}{8\,
  \Gamma\bigl(5i-4\bigr)\,\Gamma\bigl(\frac{5i}{4}\bigr)}, \label{eq:det42c} 
\\
G_{3, 0, 3, 3}(n) &= 2\prod_{i=1}^{n} \frac{
  \Gamma\bigl(6i-1\bigr)\,\Gamma\bigl(\frac{i+3}{4}\bigr)}{
  \Gamma\bigl(5i\bigr)\,\Gamma\bigl(\frac{5i-1}{4}\bigr)}, \label{eq:det42d} 
\\
G_{2, 1, 2, 0}(n) &= \prod_{i=1}^{n} \frac{
  \Gamma\bigl(6i-1\bigr)\,\Gamma\bigl(\frac{i+2}{4}\bigr)}{2(2i-1)\,
  \Gamma\bigl(5i-1\bigr)\,\Gamma\bigl(\frac{5i-2}{4}\bigr)}. \label{eq:det42e} 
\end{align}
Moreover, the following identities are conjectured to hold for all $n\geq3$:
\begin{align}
G_{3, 0, 3, 3}(n) &= \frac23 \, G_{0, 1, -2, -4}(n+1) = -\frac{1}{672} \, G_{1, 3, -2, -8}(n+1) \notag \\
&= \frac{1}{63} \, G_{5, 4, 3, -5}(n) = \frac{4}{1002001} \, G_{6, 6, 3, -9}(n) =
-\frac85 \, G_{9, 5, 8, -2}(n-1), \\
G_{1, 1, 0, -2}(n) &= -\frac{1}{49} \, G_{2, 3, 0, -6}(n) = -\frac27 \, G_{6, 4, 5, -3}(n-1) =
-\frac{4}{5577} \, G_{7, 6, 5, -7}(n-1), \\
G_{2, 1, 2, 0}(n) &= 2 \, G_{7, 4, 7, -1}(n-1).
\end{align}
\end{conjecture}
Note that \eqref{eq:det42d} is the same determinant as in~\eqref{eq:conj4j}.

{\tiny
\begin{table}
  \begin{center}
  \begin{tabular}{@{}c|l|cccccccccc}
    \multicolumn{2}{l|}{Theorem} & \ref{thm:einfach} & \ref{thm:DiFran} & \ref{thm:D3} &
    \ref{thm:D4} & \ref{thm:det22}  & \ref{thm:det33} & \ref{conj:det33x} & \ref{thm:det24} &
    \ref{thm:detx41} & \ref{conj:det42}
    \\ \hline\hline \rule{0pt}{10pt}%
    \multirow{4}{*}{$\mathfrak{I}$}
    & hol.\ rank & 2 & 3 & 3 & 3 & 3 & 4 & 4 & 3 & 3 & 5
    \\
    & degree in $n$ & 2 & 11 & 14 & 14 & 10--14 & 7--8 & 7 & 5 & 5 & 14--19
    \\
    & degree in $j$ & 4 & 11 & 12 & 12 & 7--11 & 8--17 & 9--11 & 5--8 & 9 & 10--15
    \\
    & {\tt ByteCount} & 0.03 & 0.16 & 2.12 & 1.94 & 0.08--0.17 & 0.06--0.23 & 0.26--0.51 &
      0.03--0.05 & 0.26 & 0.24--0.52
    \\
    \hline\rule{0pt}{10pt}%
    \multirow{2}{*}{\eqref{H1}}
    & order of rec. & 2 & 3 & 3 & 3 & 3 & 4 & 4 & 3 & 3 & 5
    \\
    & degree in $n$ & 3 & 15 & 25 & 25 & 12--21 & 18--26 & 21--23 & 6--10 & 11 & 49--59
    \\
    \hline \rule{0pt}{10pt}%
    \multirow{4}{*}{\eqref{H2}}
    & time 1. sum & 0.006 & 0.67 & - & - & 0.18--0.74 & 2.26--24.6 & - & 0.01--0.03 & 2.54 & -
    \\
    & time 2. sum & 0.009 & 0.71 & - & - & 0.19--0.83 & 3.36--20.8 & - & 0.01--0.04 & 7.18 & -
    \\
    & hol.\ rank & 3 & 6 & - & - & 6 & 10 & - & 7 & 7 & -
    \\
    & {\tt ByteCount} & 0.03 & 1.32 & - & - & 0.71--1.62 & 0.47--2.51 & - & 0.04--0.19 & 2.08 & -
    \\
    \hline\rule{0pt}{10pt}%
    \multirow{4}{*}{\eqref{H3}}
    & time 1. sum & 0.005 & 0.77 & 21.7 & 17.0 & 0.47--0.92 & 7.39--12.2 & - & 0.02--0.04 & 0.35 & -
    \\
    & time 2. sum & 0.011 & 0.6 & - & 65.7 & 0.35--0.71 & 3.81--8.42 & - & 0.01--0.03 & 0.97 & -
    \\
    & order of rec. & 2 & 6 & - & 6 & 6 & 10 & - & 5 & 5 & -
    \\
    & degree in $n$ & 1 & 52 & - & 75 & 45--57 & 74--93 & - & 14--22 & 24 & -
  \end{tabular}
  \end{center}
  \vskip6pt
  \caption{Computational data from the proofs by holonomic Ansatz; if there
    is more than a single determinant in a theorem, the range of values over all
    instances is displayed; if in such a situation only a single value appears, it means
    that all instances had the same value. {\tt ByteCount} refers to the homonymous {\sl Mathematica}
    command (applied to the final annihilator, not the intermediate creative telescoping
    results) and the values are given in MB. All timings are given in hours. }
  \label{tab:comp}
\end{table}
}

\section{Open questions}
\label{sec:open}

In this paper, we have proven 68 determinant evaluations that are inspired
by Di Francesco's determinant for twenty-vertex configurations (Theorems
\ref{thm:einfach}, \ref{thm:DiFran}, \ref{thm:D3}, \ref{thm:D4},
\ref{thm:det22}, 
\ref{thm:det33}, 
\ref{thm:det24}, 
\ref{thm:detx41},~\ref{thm:MS1}). We found another
21 determinants that seem to have a nice closed form, which we unfortunately
were not able to prove (Conjectures~\ref{conj:det33x}, 
\ref{conj:4j},
\ref{conj:det42}, 
and Proposition~\ref{prop:4j}). 
Also the precise description of the polynomial factor in $F_{3,0,x+3,x+3}(n)$
is left as an open problem (Conjecture~\ref{conj:MS1}). Most of these determinants
were found by computer search in certain ranges, and some of them were found to
belong to infinite families. It would be interesting to study whether the remaining
ones, which at the moment seem to be ``sporadic'' and unsystematic cases, can be
explained and characterized, and whether there are more examples outside of our search
ranges.

\medskip
Another intriguing question concerns the existence of $q$-analogues of the presented
determinant formulas.
Indeed, Theorem~\ref{thm:einfach} has the following $q$-analogue.

\begin{theorem}
For all non-negative integers $n$, we have
\begin{equation} \label{eq:qdet} 
\det_{0\le i,j\le n-1}\left(
a^i\frac {(x q^{1+i};q)_j} {(q;q)_j}
+\frac {(x q^{1-i};q)_j} {(q;q)_j}\right)
=2\,q^{-\binom n3}(-x)^{\binom n2}
\prod _{i=0} ^{n-1}(a q^{i};q)_{i},
\end{equation}
where $(\alpha;q)_p:=(1-\alpha)(1-q\alpha)\cdots(1-q^{p-1}\alpha)$
for $p\ge1$ and $(\alpha;q)_0:=1$.
\end{theorem}

\begin{proof}
We begin in the spirit of the second proof of
Theorem~\ref{thm:einfach} in Section~\ref{sec:warmup}.
In particular, we adopt the constant-term
notation from there.

Using the $q$-binomial theorem (see \cite[Eq.~(1.3.2);
  Appendix~(II.3)]{GaRaAF}) 
$$
\sum_{k=0} ^{\infty}\frac {(a;q)_k} {(q;q)_k}z^k
=
  \frac {{( \let\over/ a z;q)_\infty}}
{{( \let\over/  z;q)_\infty}},
$$
our determinant on the left-hand side of \eqref{eq:qdet} can be written as
\begin{align*}
\CT_{\mathbf z}&
\det_{0\le i,j\le n-1}
\left(z_i^{-j}\left(a^i\frac {(xz_iq^{1+i};q)_\infty} {(z_i;q)_\infty}
      +\frac {(xz_iq^{1-i};q)_\infty} {(z_i;q)_\infty}\right)\right)\\
&=\CT_{\mathbf z}
\left(
\prod _{i=0} ^{n-1}\frac {z_i^{-n+1}} {(z_i;q)_\infty}
\left(a^i {(xz_iq^{1+i};q)_\infty} 
      + {(xz_iq^{1-i};q)_\infty} \right)\right)
\det_{0\le i,j\le n-1}
\left(z_i^{n-j-1}\right).
\end{align*}
The last determinant can be evaluated by means of the evaluation of the
Vandermonde determinant. Thus, we obtain
$$
\CT_{\mathbf z}
\left(
\prod _{i=0} ^{n-1}\frac {(xz_iq^{n};q)_\infty}
      {z_i^{n-1}\,(z_i;q)_\infty}
\left(a^i {(xz_iq^{1+i};q)_{n-i-1}} 
      + {(xz_iq^{1-i};q)_{n+i-1}} \right)\right)
\prod _{0\le i<j\le n-1} ^{}(z_i-z_j)
$$
for the determinant on the left-hand side of \eqref{eq:qdet}.
Again, since this is a constant term, we get the same value if we permute the
variables $z_0,z_1,\dots, z_{n-1}$. So, we symmetrize the last
expression and get
\begin{multline*}
\frac {1} {n!}\CT_{\mathbf z}
\left(
\prod _{i=0} ^{n-1}\frac {(xz_iq^{n};q)_\infty}
      {z_i^{n-1}\,(z_i;q)_\infty}
\right)
\left(\prod _{0\le i<j\le n-1} ^{}(z_i-z_j)\right)\\
\times
\det_{0\le i,j\le n-1}
\left(a^i {(xz_jq^{1+i};q)_{n-i-1}} 
      + {(xz_jq^{1-i};q)_{n+i-1}} \right)
\end{multline*}
for our determinant.
The determinant in the above expression
is a polynomial in $xz_0,xz_1,\dots,xz_{n-1}$,
which is skew-symmetric in these quantities. Hence, it is divisible
by the Vandermonde product
$$\prod _{0\le i<j\le n-1} ^{}(xz_i-xz_j)
=x^{\binom n2}\prod _{0\le i<j\le n-1} ^{}(z_i-z_j).
$$
This shows that 
the determinant on the left-hand side of~\eqref{eq:qdet} equals
$$
\frac {x^{\binom n2}} {n!}\CT_{\mathbf z}
\left(
\prod _{i=0} ^{n-1}\frac {(xz_iq^{n};q)_\infty}
      {z_i^{n-1}\,(z_i;q)_\infty}
\right)
\left(\prod _{0\le i<j\le n-1} ^{}(z_i-z_j)^2\right)
f(xz_0,xz_1,\dots,xz_{n-1}),
$$
where $f(xz_0,xz_1,\dots,xz_{n-1})$ is some polynomial in
the given quantities.

Now, repeating arguments from the proof of Theorem~\ref{thm:einfach},
the square of the Vandermonde product, $\prod _{0\le i<j\le n-1} ^{}
\left(z_i-z_j\right)^2$, is a homogeneous polynomial of degree $n(n-1)$.
Moreover, it is not very difficult to see that the coefficient of
$(z_0z_1\cdots z_{n-1})^{n-1}$ in it equals~$(-1)^{\binom n2}n!$. 
In view of what we have found so far,
this implies that the determinant
on the left-hand side of~\eqref{eq:qdet} equals
$$
(-x)^{\binom n2}
f(0,0,\dots,0).
$$

It remains to compute the constant $f(0,0,\dots,0)$.
By inspection, the highest power of~$x$ in the determinant on the
left-hand side of~\eqref{eq:qdet} is exactly $x^{\binom n2}$.
Thus, we will obtain $(-1)^{\binom n2}f(0,0,\dots,0)$ if
we take the highest coefficient of each individual entry of this
determinant, that is, if we compute
\begin{multline*}
\det_{0\le i,j\le n-1}\left(
a^i\frac {(-1)^jq^{\binom j2}\left( q^{1+i}\right)^j} {(q;q)_j}
+\frac {(-1)^jq^{\binom j2}\left(q^{1-i}\right)^j}  {(q;q)_j}\right)\\
=(-1)^{\binom n2}
q^{\binom n3+\binom n2}a^{\frac {1} {2}\binom n2}
\left(
\prod _{j=0} ^{n-1}\frac {1} { {(q;q)_j}}\right)
\det_{0\le i,j\le n-1}\left(
\left(a^{1/2} q^{j}\right)^i
+\left(a^{1/2}q^{j}\right)^{-i}\right).
\end{multline*}
This determinant can be evaluated by means of \cite[Eq.~(2.5)]{KratBN}.
After some simplification, one obtains the desired result.
\end{proof}

However, (so far?)
we were not able to find $q$-analogues of any of the other determinant evaluations
proved or conjectured in this paper.

\medskip
Since Di Francesco's original determinant arose in combinatorics, we propose
as a future research direction to come up with combinatorial interpretations
of our ``variations on the theme''. In this regard, we report
two compelling coincidences,
where some of our {\it product formulas} appear in a seemingly unrelated 
--- combinatorial --- context. These may hint at where to look for
such combinatorial interpretations.

Namely,
in \cite{FiSAAA}, Fischer and Schreier-Aigner consider the $(-1)$-enumeration
of arrowed Gelfand--Tsetlin patterns. These are intimately related
to alternating sign matrices, and thus to configurations in the six-vertex
model. The main results in~\cite{FiSAAA} ``overlap" with two of our
results. However, what the exact relationship is, is mysterious
to us, as we now explain.

In Theorem~1 of \cite{FiSAAA} it is shown that a certain $(-1)$-enumeration
of arrowed Gelfand--Tsetlin patterns is given by
$$
2^n
\prod _{i=1} ^{n}
\frac {(m-n+3i+1)_{i+1}\,(m-n+i+1)_i}
      {\left(\frac {m-n+i+2} {2}\right)_{i-1}\,(i)_i}.
$$
It is not difficult to see that, if in this expression we replace
$m$ by $x+n-1$, then we obtain exactly the right-hand side
of~\eqref{eq:CK1} multiplied by $2^{n-1}$. Although Fischer and
Schreier-Aigner also obtain the above formula by a determinant
evaluation, the relationship with our determinant and the
enumeration of domino tilings of generalized Aztec triangles
(cf.\ the proof of Theorem~\ref{thm:D3} in Section~\ref{sec:conj2})
eludes us.

On the other hand, in Theorem~2 of \cite{FiSAAA} Fischer and Schreier-Aigner
consider another $(-1)$-enumeration of arrowed Gelfand--Tsetlin patterns
and find that it is given by
$$
3^{\binom {n+1}2}
\prod _{i=1} ^{n}\frac {(2n+m+2-3i)_i} {(i)_i}
=
3^{\binom {n+1}2}
\prod _{i=1} ^{n}\frac {(m-n+3i+2)_{n-i}} {(i)_i}.
$$
Here, visibly, if in this expression we replace
$m$ by $x+n-1$, then we obtain the right-hand side
of~\eqref{detx41} multiplied by $2^{-\binom {n}2-1}3^{n}$.
Again, Fischer and
Schreier-Aigner obtain this formula by a determinant
evaluation, which however does not help us to understand what this
has to do with our determinant.

\bibliographystyle{plain}
\bibliography{dets}

\begin{thebibliography}{10}

\bibitem{ComtAA}
Louis Comtet.
\newblock {\em Advanced combinatorics}.
\newblock D. Reidel Publishing Co., Dordrecht, enlarged edition, 1974.
\newblock The art of finite and infinite expansions.

\bibitem{CoHuKr23}
Sylvie Corteel, Frederick Huang, and Christian Krattenthaler.
\newblock Domino tilings of generalized {A}ztec triangles.
\newblock preprint, 39~pp.; {\tt arXiv:2305.01774}.

\bibitem{DiFran21}
Philippe Di~Francesco.
\newblock Twenty vertex model and domino tilings of the {A}ztec triangle.
\newblock {\em Electronic Journal of Combinatorics}, 28(4):Paper No.~4.38,
  50~pp, 2021.

\bibitem{DiFranGui}
Philippe Di~Francesco and Emmanuel Guitter.
\newblock Twenty-vertex model with domain wall boundaries and domino tilings.
\newblock {\em Electronic Journal of Combinatorics}, 27(2):Paper No. 2.13, 63,
  2020.

\bibitem{DuKoutschanThanatipanondaWong22}
Hao Du, Christoph Koutschan, Thotsaporn Thanatipanonda, and Elaine Wong.
\newblock Binomial determinants for tiling problems yield to the holonomic
  ansatz.
\newblock {\em European Journal of Combinatorics}, 99:103437, 2022.

\bibitem{FiSAAA}
Ilse Fischer and Florian Schreier-Aigner.
\newblock $(-1)$-enumerations of arrowed {G}elfand--{T}setlin patterns.
\newblock preprint, 21~pp.; {\tt arXiv:2302.04164}.

\bibitem{GaRaAF}
George Gasper and Mizan Rahman.
\newblock {\em Basic hypergeometric series}.
\newblock Encyclopedia of Mathematics And Its Applications~96, Cambridge
  University Press, Cambridge, 2nd edition edition, 2004.

\bibitem{GospAB}
Ralph~William Gosper.
\newblock Decision procedure for indefinite hypergeometric summation.
\newblock {\em Proceedings of the National Academy of Sciences of the United
  States of America}, 75:40--42, 1978.

\bibitem{IshikawaKoutschan12}
Masao Ishikawa and Christoph Koutschan.
\newblock Zeilberger's holonomic ansatz for {P}faffians.
\newblock In {\em Proceedings of the International Symposium on Symbolic and
  Algebraic Computation (ISSAC)}, pages 227--233, New York, USA, 2012. ACM.

\bibitem{Kauers09}
Manuel Kauers.
\newblock Guessing handbook.
\newblock Technical Report 09-07, RISC Report Series, Johannes Kepler
  University, Linz, Austria, 2009.
\newblock {\tt
  http:/$\!$/www.risc.jku.at/\linebreak[0]research/\linebreak[0]combinat/\linebreak[0]software/\linebreak[0]Guess/}.

\bibitem{Koutschan10b}
Christoph Koutschan.
\newblock {HolonomicFunctions} (user's guide).
\newblock Technical Report 10-01, RISC Report Series, Johannes Kep\-ler
  University, Linz, Austria, 2010.
\newblock {\tt
  https:/$\!$/risc.jku.at/\linebreak[0]sw/\linebreak[0]holonomicfunctions/}.

\bibitem{EM}
Christoph Koutschan.
\newblock Electronic material accompanying the article ``{D}eterminant
  evaluations inspired by {D}i {F}rancesco's determinant for twenty-vertex
  configurations'', 2024.
\newblock Available at \url{http://www.koutschan.de/data/det3/}.

\bibitem{KoutschanKauersZeilberger11}
Christoph Koutschan, Manuel Kauers, and Doron Zeilberger.
\newblock Proof of {G}eorge {A}ndrews's and {D}avid {R}obbins's $q$-{TSPP}
  conjecture.
\newblock {\em Proceedings of the National Academy of Sciences},
  108(6):2196--2199, 2011.

\bibitem{KoutschanNeumuellerRadu16}
Christoph Koutschan, Martin Neum\"uller, and Cristian-Silviu Radu.
\newblock Inverse inequality estimates with symbolic computation.
\newblock {\em Advances in Applied Mathematics}, 80:1--23, 2016.

\bibitem{KoutschanThanatipanonda13}
Christoph Koutschan and Thotsaporn Thanatipanonda.
\newblock Advanced computer algebra for determinants.
\newblock {\em Annals of Combinatorics}, 17(3):509--523, 2013.

\bibitem{KoutschanThanatipanonda19}
Christoph Koutschan and Thotsaporn Thanatipanonda.
\newblock A curious family of binomial determinants that count rhombus tilings
  of a holey hexagon.
\newblock {\em Journal of Combinatorial Theory, Series~A}, 166:352--381, 2019.

\bibitem{KratBI}
Christian Krattenthaler.
\newblock An alternative evaluation of the {A}ndrews--{B}urge determinant.
\newblock In {\em Mathematical Essays in Honor of Gian-Carlo Rota}, pages
  263--270, Boston, USA, 1998. Progress in Math., vol.~161, Birkh\"auser.

\bibitem{KratBN}
Christian Krattenthaler.
\newblock Advanced determinant calculus.
\newblock {\em S\'eminaire Lotharingien Combinatoire}, 42:Article~B42q, 67~pp,
  1999.

\bibitem{PaScAA}
Peter Paule and Markus Schorn.
\newblock A {{\sl Mathematica}} version of {Z}eilberger's algorithm for proving
  binomial coefficient identities.
\newblock {\em Journal of Symbolic Computation}, 20(5-6):673--698, 1995.

\bibitem{PetkovsekWilfZeilberger96}
Marko Petkov\v{s}ek, Herbert~Saul Wilf, and Doron Zeilberger.
\newblock {\em {$A=B$}}.
\newblock A.~K.~Peters, Ltd., Wellesley MA, 1996.

\bibitem{SlatAC}
Lucy~Joan Slater.
\newblock {\em Generalized hypergeometric functions}.
\newblock Cambridge University Press, Cambridge, 1966.

\bibitem{Zeilberger90a}
Doron Zeilberger.
\newblock A fast algorithm for proving terminating hypergeometric identities.
\newblock {\em Discrete Mathematics}, 80(2):207--211, 1990.

\bibitem{Zeilberger91}
Doron Zeilberger.
\newblock The method of creative telescoping.
\newblock {\em Journal of Symbolic Computation}, 11:195--204, 1991.

\bibitem{Zeilberger07}
Doron Zeilberger.
\newblock The holonomic ansatz {II}. {A}utomatic discovery(!) and proof(!!) of
  holonomic determinant evaluations.
\newblock {\em Annals of Combinatorics}, 11(2):241--247, 2007.

\end{thebibliography}

\end{document}